\documentclass[12pt, reqno, a4paper, oneside]{amsart}
\usepackage{graphicx, epsfig, psfrag, subfigure}
\usepackage{amsfonts,amsmath,amssymb,amsbsy,amsthm}
\usepackage{bm}
\usepackage{color }
\usepackage{float}
\usepackage{mathrsfs}
\usepackage[left=3cm,right=3cm,top=3cm,bottom=3cm]{geometry}
\usepackage{longtable}
\usepackage{todonotes}
\usepackage{epsfig}

\usepackage{environments}
\numberwithin{theorem}{section}
\numberwithin{equation}{section}
\numberwithin{definition}{section}

\renewcommand{\cases}[1]{\left\{ \begin{array}{rl} #1 \end{array} \right.}

\def\XXint#1#2#3{{\setbox0=\hbox{$#1{#2#3}{\int}$ }
\vcenter{\hbox{$#2#3$ }}\kern-.6\wd0}}

\usepackage{accents}
\newlength{\dhatheight}
\newlength{\dtildeheight}

\def\b{\big}

\def\sep{\,|\,}
\def\bsep{\,\b|\,}

\def\diam{{\rm diam}}

\def\supp{{\rm supp}}

\def\R{\mathbb{R}}

\def\Z{\mathbb{Z}}

\def\dx{\,{\rm d}x}
\def\dy{\,{\rm d}y}

\def\dd{\,{\rm d}}
\def\pp{\partial}

\def\<{\langle}
\def\>{\rangle}

\def\mA{{\sf A}}

\def\mF{{\sf F}}

\def\bfa{{\bm a}}
\def\bfg{{\bm g}}

\def\bfh{{\bm h}}

\def\D{\nabla}
\def\del{\delta}
\def\ddel{\delta^2}

\def\a{{\rm a}}
\def\c{{\rm c}}
\def\ac{{\rm ac}}
\def\i{{\rm i}}

\def\L{\Lambda}

\def\Is{\mathcal{I}}

\def\As{\mathcal{A}}
\def\Cs{\mathcal{C}}

\def\E{\mathcal{E}}

\def\tilu{\tilde u}

\def\T{\mathcal{T}}

\def\U{\mathcal{U}}
\def\Uss{\dot{\mathcal{U}}^{1,2}}

\definecolor{cocol}{rgb}{0.7, 0, 0}
\definecolor{hwcol}{rgb}{0, 0.7, 0}
\definecolor{adcol}{rgb}{0, 0, 0.7}

\begin{document}

	\title{Coupling Atomistic, Elasticity and Boundary Element Models}
	
	\author{A. S. Dedner}
	\address{A. Dedner\\ Mathematics Institute \\ Zeeman Building \\
		University of Warwick \\ Coventry CV4 7AL \\ UK}
	\email{a.s.dedner@warwick.ac.uk}

	\author{C. Ortner}
	\address{C. Ortner\\ Mathematics Institute \\ Zeeman Building \\
		University of Warwick \\ Coventry CV4 7AL \\ UK}
	\email{c.ortner@warwick.ac.uk}
	
	\author{H. Wu}
	\address{H. Wu\\ Mathematics Institute \\ Zeeman Building \\
		University of Warwick \\ Coventry CV4 7AL \\ UK}
	\email{huan.wu@warwick.ac.uk}

	\date{\today}
	\thanks{HW was supported by MASDOC doctoral training centre, EPSRC grant
		EP/H023364/1. CO was supported by ERC Starting Grant 335120.}

	\begin{abstract}
		We formulate a new atomistic/continuum (a/c) coupling scheme that employs
		the boundary element method (BEM) to obtain an improved far-field
		boundary condition. We establish sharp error bounds in a 2D model problem
		for a point defect embedded in a homogeneous crystal.
	\end{abstract}
	
	\maketitle
	
	\section{Introduction}
	Atomistic-to-continuum (a/c) coupling is a class of multi-scale methods that
	couple atomistic models with continuum elasticity models to reduce computational cost while preserving a significant level of accuracy. In
	the continuum model coarse finite element methods are often used. We refer
	to \cite{acta} and the references therein for a comprehensive introduction and a framework for error analysis.
	
	The present work explores the feasibility and effectiveness of employing
	boundary elements in addition to the existing a/c framework to better
	approximate the far-field energy which is most typically truncated.
	Specifically we combine a quasi-nonlocal (QNL) type method with a BEM, in a 2D
	model problem.
	
	
	The QNL-type coupling, first introduced in \cite{Shimokawa:2004,E:2006}, is an energy-based a/c
	method that introduces a interface region between the atomistic and continuum model
	so that the model is ``free of ghost-forces'' (a notion of consistency related
	to the patch test, see \S \ref{sec:g23_coupling}). The first explicit construction of such schemes for
	two-dimensional domains with corners is developed in \cite{PRE-ac.2dcorners} for
	a neareast-neighbour many-body site potential. We call this coupling scheme ``G23''
	for future reference. An error analysis of the G23 coupling equipped with coarse finite elements of order two or higher is described in \cite{2016-qcp2}.
	
	The boundary element method is a numerical method for solving linear partial differential equations by discretising the boundary integral formulation. For a general introduction and analysis we refer to \cite{Olaf:BEM}. In the present work we first approximate a nonlinear elasticity model by a quadratic energy functional which is then discretised by the BEM.
	
	The idea of employing a BEM-like scheme to model the elastic far-field is not
	new. For example, in \cite{Kanzaki, Sinclair} an atomistic Green's function
	method is employed to determine a far-field boundary condition which yields a
	{\em sequential} multi-scale scheme, while \cite{Woodward, Li:lattice_green}
	formulate {\em concurrent} multi-scale schemes coupling atomistic mechanics to a
	Green's function method. In this setting, a preliminary error analysis can
	already be found in \cite{EhrOrtSha:2013}. By contrast, our new scheme employs a
	BEM, i.e., a continuum elasticity Green's function approach to model the elastic
	far-field. Moreover, our formulation allows a seemless transition between
	atomistic mechanics, nonlinear continuum mechanics (FEM) and linearised
	continuum mechanics (BEM). This flexibility is particularly interesting for an
	error analysis since we are able to determine quasi-optimal error balancing
	between the two difference approximations.
	
	To conclude the introduction we remark that the BEM far-field boundary
	condition can of course be employed for other A/C coupling schemes as well
	as more complex (in terms of geometry and interaction law) atomistic
	models, but in particular the latter generalisation requires some additional
	work. With this in mind, the present work may be considered a proof of
	concept.


	\subsection{Outline}
	%
	In the present work we estimate the accuracy of a QNL-type atomistic/continuum
	coupling method employing a P1 FEM in the continuum region and P0 BEM on the
	boundary against an exact solution obtain from a fully atomistic model.
	We review the atomistic model in \S~\ref{sec:atomistic-model}, the QNL coupling scheme in \S~\ref{sec:g23_coupling} and \S~\ref{sec:fem}, and the modification to
	incorporate a BEM for the elastic far-field in \S~\ref{sect:model}.
	In \S \ref{sect:prelim} we collect notation, assumptions and preliminary results
	required to state the main results in \S~\ref{sec:main_results}.
	We then deduce the optimal approximation parameters (atomistic region size,
	continuum region size, FEM and BEM meshes) in \S \ref{sec:opt_param}. We will
	conclude that omitting the FEM region entirely yields the best possible
	convergence rate.

	\section{Method Formulation}\label{sec:method_form}

	\subsection{Atomistic model} \label{sec:atomistic-model}
	
	In order to employ the G23 coupling in \cite{PRE-ac.2dcorners}, we follow the same model construction therein. We consider an infinite 2D triangular lattice as our model geometry,
	\begin{equation*}
	\L := \mA \Z^2, \quad \text{with } \mA =
	\begin{pmatrix}
	1 & \cos(\pi /3) \\ 0 & \sin(\pi/3)
	\end{pmatrix}.
	\end{equation*}
	We define the six nearest-neighbour lattice directions by $a_1 :=(1,0)$, and
	$a_j := Q_6^{j-1}a_1, j \in \Z$, where $Q_6$ denotes the rotation through the
	angle $\pi/3$.  We equip $\L$ with an {\em atomistic triangulation}, as shown
	in Figure~\ref{fig:lattice}, which will be used in both error analysis and
	numerical simulations. We denote this triangulation by $\T$ and its elements by
	$T\in \T$. In addition, we denote $\bfa := (a_j)_{j=1}^6$, and
	$\mF \bfa : = (\mF a_j)_{j=1}^6$, for $\mF \in \R^{m\times 2}$.
	
	\begin{figure}
		\centering
		\includegraphics[width=0.5\linewidth]{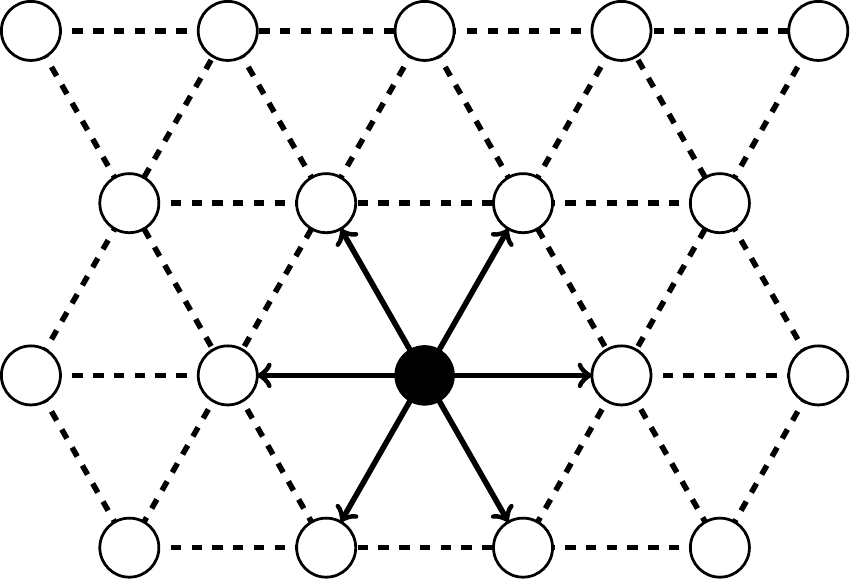}
		\caption{The lattice (circles), its canonical triangulation (dashed lines) and the six nearest-neighbour directions (arrows). This illustration is taken from \cite{2016-qcp2}.}
		\label{fig:lattice}
	\end{figure}
	
	We identify a discrete displacement map $u: \L \rightarrow \R$ with its
	continuous piecewise affine interpolant, with weak derivative $\D u$, which is
	also the pointwise derivative on each element $T \in \mathcal{T}$. For $m =
	1,2,3$, we define the spaces of displacements as
	\begin{equation*}
	\begin{aligned}
	\mathcal{U}_0 & := \b\{u \sep \L \rightarrow \R^m : \supp(\D u ) \text{ is compact} \b \}, \quad \text{and}\\
	\dot{\mathcal{U}}^{1,2}  &:= \b\{u \sep \L \rightarrow \R^m : \D u  \in L^2\b\} .
	\end{aligned}
	\end{equation*}
	We equip $\dot{\mathcal{U}}^{1,2}$ with the $H^1$-semi norm and denote $\|u\|_{\mathcal{U}^{1,2}} := \|\D u\|_{L^2(\R^2)}$. From \cite{OrShSu:2012} we know that $\mathcal{U}_0$ is dense in $\dot{\mathcal{U}}^{1,2}$ in the sense that, if $u \in \dot{\mathcal{U}}^{1,2}$, then there exist $ u_j \in \mathcal{U}_0$ such that $\D u_j \rightarrow \D u$ strongly in $L^2$.
	
	A \emph{homogeneous displacement} is a map $u_\mF: \L \rightarrow \R^m, u_\mF(x) : = \mF x$, where $\mF\in \R^{m\times2}$.
	
	For a map $u:\L \rightarrow \R^m$, we define the finite difference operator
	\begin{equation}\label{def:diff_op}
	\begin{aligned}
	D_j u(x) &:= u(x+a_j)-u(x), \quad x \in \L, j \in \{1,2,...,6\}, \quad \text{and}\\
	Du(x) &:= (D_j u(x))_{j=1}^6.
	\end{aligned}
	\end{equation}
	Note that $Du_{\mF}(x) = \mF \bfa$.
	
	We assume that the atomistic interaction is represented by a nearest-neighbour
	many-body site energy potential $V \in C^r(\R^{m\times 6})$,$r\ge 5$, with
	$V(\mathbf{0}) = 0$
	and $\nabla^j V \in L^\infty(\R^{m\times 6})$ for $j = 2, \dots, 5$.
	In addition, we assume that $V$ satisfies the \emph{point symmetry}
	\begin{equation*}
	V((-g_{j+3})_{j=1}^6) = V(\bfg) \quad \forall \bfg \in \R^{m\times 6}.
	\end{equation*}
	Because $V(\mathbf{0}) = 0$, the energy of a displacement $u\in \mathcal{U}_0$
	\begin{equation*}
	\E^\a(u): = \sum_{\ell\in \Lambda}V(Du(\ell)),
	\end{equation*}
	is well-defined.  We need the following lemma to extend $\E^\a$ to
	$\dot{\mathcal{U}}^{1,2}$ to formulate a variational problem in the energy space
	$\dot{\mathcal{U}}^{1,2}$, 
	
	\begin{lemma}
		$\E^\a :(\mathcal{U}_0,\|\D \cdot\|_{L^2} )\rightarrow \R$ is continuous and
		has a unique continuous extension to $\dot{\mathcal{U}}^{1,2}$, which we still
		denote by $\E^\a $. Moreover, the extended
		$\E^\a :(\dot{\mathcal{U}}^{1,2},\|\D \cdot\|_{L^2} ) \rightarrow \R$ is
		$r$-times continuously Fr\'{e}chet differentiable.
	\end{lemma}
	\begin{proof}
		See Lemma 2.1 in \cite{EhrOrtSha:2013}.
	\end{proof}
	
	We model a point defect by including an external potential
	$f \in C^r(\dot{\mathcal{U}}^{1,2})$ with $\partial_{u(\ell)} f(u) = 0$ for all
	$|\ell| \geq R_f$, where $R_f$ is the defect core radius,
	and $f(u + c) = f(u)$ for all constants $c$. For instance, we can think of $f$
	modelling a substitutional impurity. See also \cite{2014-bqce, Or:2011a} for
	similar approaches.
	
	Then we seek the solution to
	\begin{equation}\label{eq:y_a}
	u^\a \in \arg \min \b\{ \E^\a(u) - f(u)  \sep u \in \dot{\mathcal{U}}^{1,2} \b\}.
	\end{equation}

	For $u, \varphi,\psi \in \dot{\mathcal{U}}^{1,2}$ we define the \emph{first and
		second variations} of $\mathcal{E}^\a$ by
	\begin{equation*}
	\begin{aligned}
	\langle \delta \mathcal{E}^\a(u), \varphi \rangle &:= \lim_{t\rightarrow 0}t^{-1}\left(\mathcal{E}^\a (u+t\varphi)-\mathcal{E}^\a(u)\right), \\
	\<\del^2 \E^\a (u) \varphi, \psi\> &:=\lim_{t\rightarrow 0}t^{-1}\left(\<\del\mathcal{E}^\a (u+t\varphi),\psi\>-\<\del \mathcal{E}^\a(u), \psi\>\right).
	\end{aligned}
	\end{equation*}
	We define analogously all energy functionals introduced in later sections.
	
	\subsection{GR-AC coupling}\label{sec:g23_coupling}
	
	The Cauchy--Born strain energy function \cite{acta, E:2005a}, corresponding to the interatomic
	potential $V$ is
	\begin{equation*}
	W(\mathsf{F}): = \frac{1}{\Omega_0} V(\mathsf{F}\bfa), \qquad \text{for }
	\mathsf{F} \in \mathbb{R}^{m \times 2},
	\end{equation*}
	where $\Omega_0 := \sqrt{3}/2$ is the volume of a unit cell of the lattice $\Lambda$. Hence $W(\mathsf{F})$ is the energy per volume of the homogeneous lattice $\mathsf{F}\Lambda$.
	It is shown in \cite{Hudson:stab} that, in a triangular lattice with anti-plane elasticity,
	$\D ^2 W({\bf{0}}) = \mu I^{2\times 2}$ for some constant $\mu > 0$
	(the shear modulus), which will be used in the formulation of BEM in later sections.

	Let $\As \subset \L$ be the set of all lattices sites for which we require
	full atomistic accuracy. We define the set of interface lattice sites as
	\begin{equation*}
	\Is := \b\{ \ell \in \L \setminus \As \bsep \ell +a_j \in \As \text{ for some } j \in \{1,\dots,6\} \b \}
	\end{equation*}
	and we define the remaining lattice sites as
	$\Cs : = \L \setminus (\As \cup \Is)$. Let $\Omega_\ell$ be the Voronoi cell
	associated with site $\ell$. We define the continuum region
	$\Omega^\c : = \R^2 \setminus \bigcup_{\ell \in \As \cup \Is} \Omega_{\ell}$;
	see Figure \ref{fig:decomp}. We also define $\Omega^\a $ and $\Omega^\i$ analogously.

	\begin{figure}
		\centering
		\includegraphics[width=0.65\linewidth]{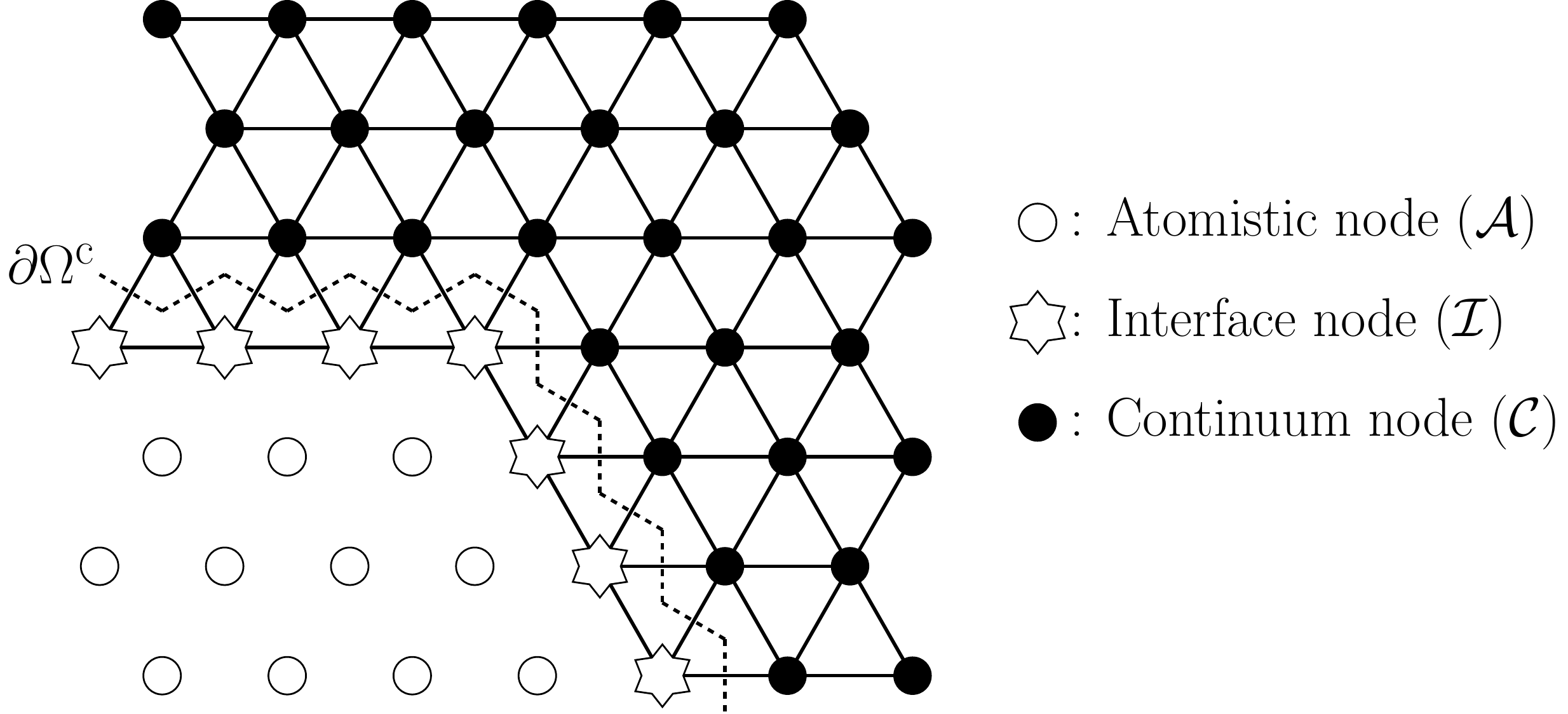}
		\caption{\small{The domain decomposition with a layer of interface atoms. This illustration is taken from \cite{2016-qcp2}.}}
		\label{fig:decomp}
	\end{figure}

	A general form for the GRAC-type a/c coupling energy \cite{E:2006,
		PRE-ac.2dcorners} is
	\begin{equation}\label{def:ac_general}
	\E^{\rm ac}(u) = \sum_{\ell\in \As} V(Du(\ell)) + \sum_{\ell\in \Is} V\left((\mathcal{R}_\ell D_ju(\ell))_{j=1}^6\right)+\int_{\Omega^\c} W(\D u(x)) \dx,
	\end{equation}
	where $\mathcal{R}_\ell D_j u(\ell): = \sum_{i=1}^6C_{\ell,j,i}D_i u(\ell)$.
	The parameters $C_{\ell,j,i}$ are determined such that the coupling
	scheme satisfies the ``patch tests'':
	
	$\mathcal{E}^{\rm ac}$ is \emph{locally energy consistent} if, for all $\mF \in \mathbb{R}^{m\times2}$,
	\begin{equation}\label{econs}
	V_\ell^i(\mF\bfa) = V(\mF\bfa) \quad \forall \ell \in \mathcal{I}.
	\end{equation}
	
	$\mathcal{E}^{\rm ac}$ is \emph{force consistent} if, for all $\mF \in \mathbb{R}^{m\times2}$,
	\begin{equation}\label{fcons}
	\del \E^\ac(u_\mF) = 0, \quad \text{where} \quad u_\mF(x) : = \mF x.
	\end{equation}
	
	$\mathcal{E}^{\rm ac}$ is \emph{patch test consistent} if it satisfies both \eqref{econs} and \eqref{fcons}.
	\medskip
	
	For simplicity we write
	$$V^i_\ell(Du(\ell)) : = V\left((\mathcal{R}_\ell D_ju(\ell))_{j=1}^6\right).$$
	Following \cite{PRE-ac.2dcorners} we make the following standing assumption (see
	Figure \ref{fig:interface_corners} for examples).
	
	{\bf (A0)} \emph{Each vertex $\ell \in \Is$ has exactly two neighbours in $\Is$, and at least one neighbour in $\Cs$.}
	
	\begin{figure}
		\centering
		\includegraphics[width=0.9\linewidth]{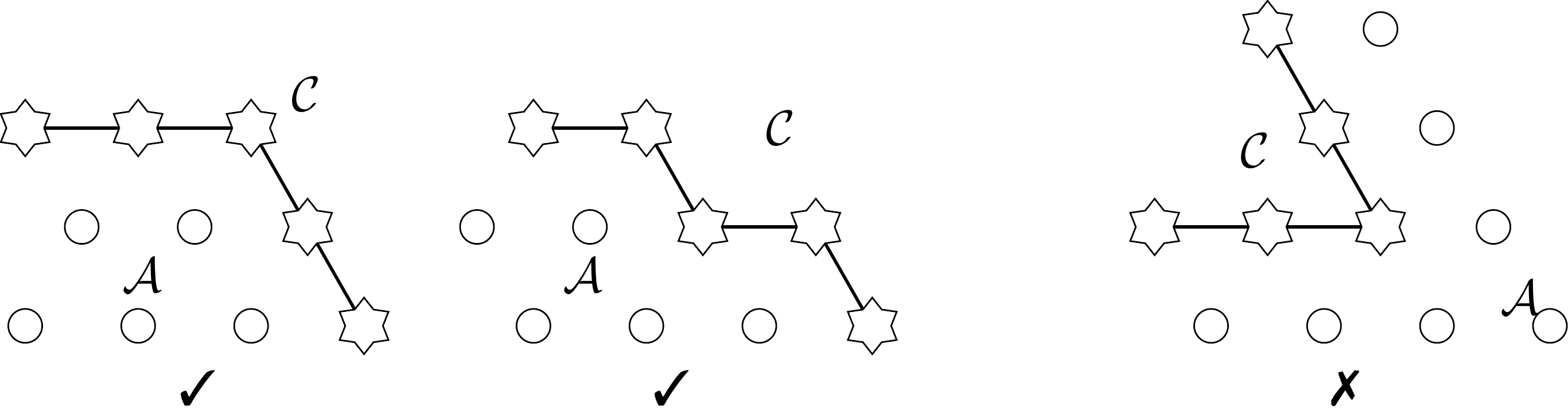}
		\caption{\small The first two configurations are allowed. The third configuration is not allowed as the interface atom at the corner has no nearest neighbour in the continuum region, and should instead be taken as an atomistic site. This illustration is taken from \cite{2016-qcp2}.}
		\label{fig:interface_corners}
	\end{figure}

	Under this assumption, the geometry reconstruction operator $\mathcal{R}_\ell$ is then defined by
	\begin{equation*}
	\begin{aligned}
	\mathcal{R}_\ell D_j y(\ell) &:= (1-\lambda_{\ell,j}) D_{j-1}y(\ell) + \lambda_{\ell,j} D_{j}y(\ell) + (1-\lambda_{\ell,j}) D_{j+1}y(\ell),\\
	\lambda_{x,j} & :=
	\cases{2/3, & x+a_j \in \Cs \\
		1, &\text{otherwise} };
	\end{aligned}
	\end{equation*}
	see Figure \ref{fig:general_interface}. The resulting a/c coupling method is called G23 and the corresponding energy functional $\E^{\rm g23}$. It is proven in \cite{PRE-ac.2dcorners} that this choice of coefficients (and only this choice) leads to patch test consistency \eqref{econs} and \eqref{fcons}.
	
	\begin{figure}
		\centering
		\includegraphics[width=0.55\linewidth]{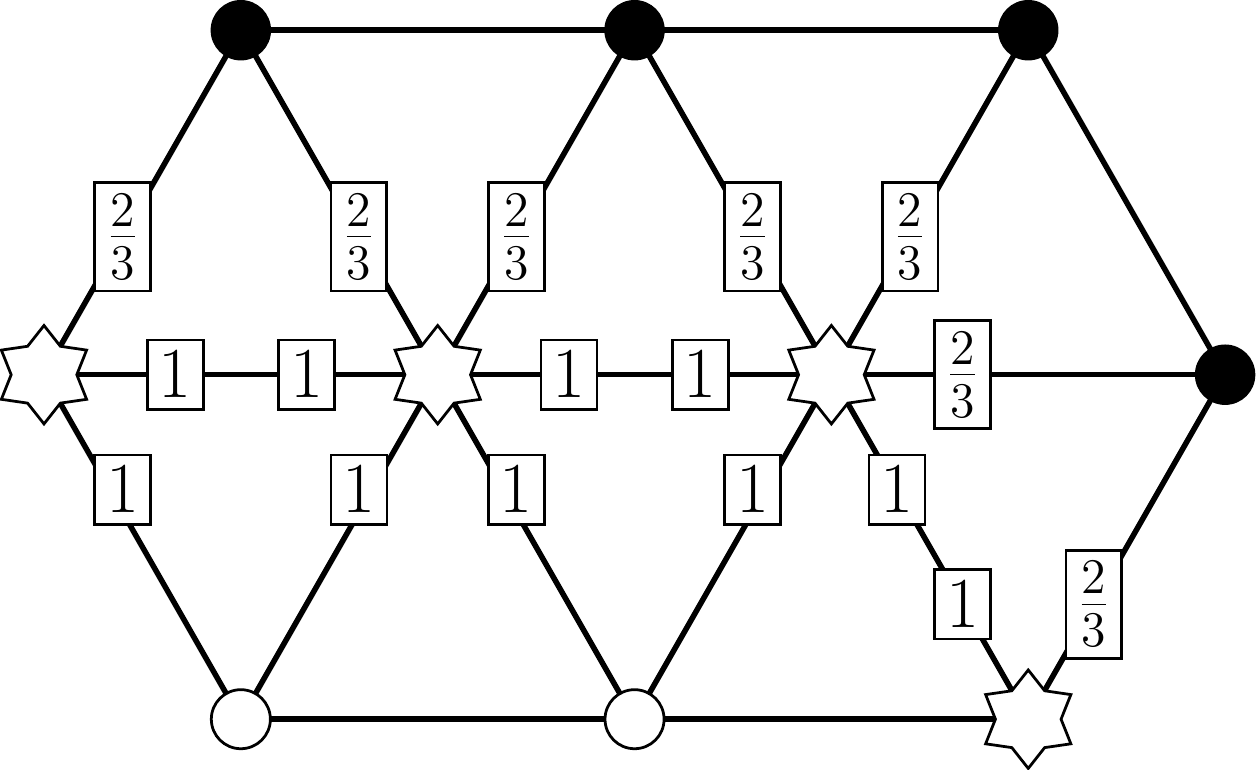}
		\caption{\small The geometry reconstruction coefficents $\lambda_{x,j}$ at the interface sites. This illustration is taken from \cite{2016-qcp2}.}
		\label{fig:general_interface}
	\end{figure}
	
	For future reference we decompose the canonical triangulation $\T$ as follows:
	\begin{equation}\label{def:ac}
	\begin{aligned}
	\mathcal{T}_\mathcal{A} :& = \{ T \in\mathcal{T} \,|\, T\cap(\mathcal{I}\cup \mathcal{C}) = \emptyset, \}, \\
	\mathcal{T}_\mathcal{C} :& = \{ T \in\mathcal{T}\,|\, T\cap(\mathcal{I}\cup \mathcal{A}) = \emptyset, \} \quad \text{and} \\
	\mathcal{T}_\mathcal{I}:& =\mathcal{T}\setminus (\mathcal{T}_\mathcal{C}\cup\mathcal{T}_\mathcal{A}).
	\end{aligned}
	\end{equation}
	
	\subsection{The finite element scheme} \label{sec:fem}
	In the atomistic region $\Omega^\a $ and the interface region $\Omega^\i $, the interactions are represented by
	discrete displacement maps, which are identified with their linear
	interpolant. In these regions there is no approximation error.
	
	On the other hand, as formulated in \eqref{def:ac_general}, the interactions are
	approximated by the Cauchy--Born energy in the continuum region $\Omega^\c_h$.
	
	Let $K>0$ be the inner radius of the atomistic region,
	\begin{equation*}
	K := \sup\b \{ r>0 \,|\, \mathcal{B}_r \cap \L \subset \As    \b \},
	\end{equation*}
	where $\mathcal{B}_r$ denotes the ball of radius $r$ centred at $0$.  We assume throughout that
	$K \geq R_f$ to ensure that the defect core is contained in the atomistic region.
	
	Let $\Omega_h$ be the entire computational domain and $N>0$ be the inner radius of $\Omega_h$, i.e.,
	\begin{equation*}
	N : = \sup \b \{r>0 \,|\, \mathcal{B}_r  \subset \Omega_h \b \}.
	\end{equation*}
	Let $\T _h $ be a finite element triangulation of $\Omega_h$ which satisfies, for $T \in \T_h $,
	\begin{equation*}
	T \cap (\As \cup \Is) \ne \emptyset \quad \Rightarrow \quad T \in  \T.
	\end{equation*}
	In other words, $\T _h$ and $\T $ coincide in the atomistic and interface regions, whereas in the continuum region the mesh size may increase
	towards the domain boundary.
	
	We observe that the concrete construction of $\T_h$ will be based on the choice of
	the domain parameters $K$ and $N$; hence we
	will write $\T_h(K,N)$ to emphasize this dependence. To eliminate the possibility of extreme angles on elements, we assume throughout that the family
	$(\T_h(K,N))_{K, N}$ is \emph{uniformly shape-regular}, i.e., there exists $c>0$
	such that,
	\begin{equation}
	\label{eq:unif-shape-reg}
	\diam  (T) ^2 \le c |T| , \quad \forall T\in \T_h(K, N), \forall K \le N,
	\end{equation}
	and that the induced mesh on $\Gamma_h := \pp \Omega_h$ is uniformly
	quasi-uniform.
	
	Hence in the analysis we can avoid deteriorated constants in finite element interpolation error
	estimates. In later sections we will again drop the parameters from the notation
	by writing $\T_h \equiv \T_h(K, N)$ but implicitly will always keep the
	dependence.

	Similar to \eqref{def:ac}, we denote the atomistic, interface and continuum
	elements by $\mathcal{T}_h^a,\mathcal{T}_h^i$ and $\mathcal{T}_h^c$,
	respectively. We observe that $\mathcal{T}_h^a = \mathcal{T}_\mathcal{A}$ and
	$\mathcal{T}_h^i = \mathcal{T}_\mathcal{I}$. We also let $\mathcal{N}_h$ be
	the number of degrees of freedom of $\T _h$.
	
	We define the finite element space of admissible displacements as
	\begin{equation}
	\label{eq:defn-Uh}
	\begin{aligned}
	\mathcal{U}_h : = \b\{u\in C(\R^2; \R^m) \sep\, u|_T \in \mathbb{P}^1(T) \text{ for } T\subset \mathcal{T}_h\}.
	\end{aligned}
	\end{equation}

\subsection{GR-AC coupling with BEM}\label{sect:model}
In \cite{2016-qcp2}, we employed finite element methods to approximate the solution. We applied P2-FEM with Dirichlet boundary conditions. To improve the far-field description, we now consider applying a boundary element method to approximate the far-field energy.

Recall that the general form (\ref{def:ac_general}) of the GR-AC type coupling energy is
\begin{equation*}
\E^{\rm ac}(u) = \sum_{\ell\in \As} V(Du(\ell)) + \sum_{\ell\in \Is} V^\i_\ell((Du(\ell)))+\int_{\Omega^\c} W(\D u(x)) \dx.
\end{equation*}

In the far-field we can approximate the Cauchy--Born energy by the
 linearization (recall that $\nabla^2 W(0) = \mu I_{2\times 2}$)
\begin{equation}\label{eq:E_lin}
\begin{aligned}
\E_{\rm lin}^{\rm ac}(u) &= \sum_{\ell\in \As} V(Du(\ell)) + \sum_{\ell\in \Is} V^\i_\ell((Du(\ell)))+\int_{\Omega_h^\c} W(\D u(x)) \dx + \int_{\R ^2\setminus \Omega_h}\frac{\mu}{2}|\D u|^2\\
&=: \E ^{\rm ac}_h(u) + \int_{\R ^2\setminus \Omega_h}\frac{\mu}{2}|\D u|^2.
\end{aligned}
\end{equation}

We seek the minimizer of above energy functional
\begin{equation*}
u^* : = \arg \min \{\E^{\rm ac}_{\rm lin} (u) -f(u)\,:\, u \in \dot{\mathcal{U}}^{1,2} \}.
\end{equation*}

For numerical simulations, we exploit the boundary integral to represent the
quadratic term $\int_{\R ^2\setminus \Omega_h}\frac{\mu}{2}|\D u|^2$.

In preparation, let $\Gamma_h : = \pp \Omega_h$, $\gamma^{\rm int}_0 : C(\Omega_h) \rightarrow C(\Gamma_h)$ and $\gamma^{\rm ext}_0 : C(\Omega_h^\complement) \rightarrow C(\Gamma_h)$ be the interior and exterior trace operators respectively, then we define
\begin{equation}\label{eq:Etot}
	\E^{\rm ac}_*(u) : = \E^{\rm ac}_h(u) +
	\inf_{ \gamma^{\rm ext}_0v = \gamma^{\rm int}_0u } \frac{\mu}{2} \int_{\Omega_h^{\complement}}|\D v|^2.
\end{equation}
Let
\begin{equation*}
\bar{u}_h : = \arg \min \{\E ^{\rm ac}_*(u): u\in \mathcal{U}_h\}
\end{equation*}
and
\begin{align}\label{def:v_inf}
v_h &:=\arg\min \left\{\int_{\Omega_h^\complement}|\D v|^2 \,:\, v\in \dot{H}^1(\Omega_h^\complement) , \gamma^{\rm ext}_0 v = \gamma^{\rm  int}_0\bar{u}_h \right\},
\\ \notag
u_h^* &:= \arg \min \left\{ \E ^{\rm ac}_{\rm lin} (u)-f(u) \,:\, u \in (\mathcal{U}_h +\dot{H}^1(\Omega_h^\complement)) \cap \dot{H}^1(\R^2) \right\},
\end{align}
then clearly $u_h^* = \bar{u}_h$ in $\Omega_h$ while $u_h^* = v_h$ in $\Omega_h^\complement$.
The inf-problem \eqref{def:v_inf} can be expressed as an exterior Laplace problem
\begin{equation}\label{eq:lap}
\begin{aligned}
-\Delta v &= 0, \quad \text{in } \Omega_h^{\complement}\\
v & = \gamma^{\rm int}_0\bar{u}_h, \quad \text{on } {\Gamma_h} \\
|v(x)-u_0| & = \mathcal{O}\left({\frac{1}{|x|}}\right) \quad \text{as }|x|\rightarrow \infty,
\end{aligned}
\end{equation}
where $u_0$ is a constant determined by the inner boundary condition $v  = \gamma^{\rm int}_0\bar{u}_h$ on $\Gamma_h$. This exterior Laplace problem can be solved by boundary integrals and be approximated by boundary element methods.

\subsubsection{Boundary integrals}

In this section, we formally outline how we combine the BEM with a/c coupling. Technical details will be presented in later sections. For a complete introduction to BEM we refer to \cite{Olaf:BEM}.

To define Sobolev spaces of fractional order, we use the Slobodeckij semi-norm.
\begin{definition}
	Let $\Gamma \subset \R ^d$ be a Lipschitz boundary, then for $0<s<1$, we define
	\begin{align*}
	|v|_{H^s(\Gamma)}& := \left(\int_\Gamma\int_{\Gamma} \frac{[v(x)-v(y)]^2}{|x-y|^{d-1+2s}}{\rm d} S(x) {\rm d} S(y)\right)^{1/2}, \\
	\|v\|_{H^s(\Gamma)}&: =  \left(\|v\|^2_{L^2(\Gamma)} + |v|^2_{H^s(\Gamma)} \right)^{1/2}, \quad \text{and}\\
	H^s_{(\Gamma)} &: = \big\{ u\in L^2(\Gamma)\, | \, |v|_{H^s(\Gamma)} < \infty\big\}.
	\end{align*}
	For $0<s<1$, $H^{-s}(\Gamma)$ is defined as the dual space of $H^{s}(\Gamma)$:
	\begin{equation*}
	\|v\|_{H^{-s}(\Gamma)} :  = \sup_{0\neq w \in H^{s}(\Gamma)}\frac{\<v, w\>_{\Gamma}}{\|w\|_{H^{s}(\Gamma)}},
	\end{equation*}
	with respect to the duality pairing
	\begin{equation*}
	\<v, w\>_{\Gamma} := \int_{\Gamma} v(x)w(x)\dx .
	\end{equation*}
\end{definition}


Using the Trace Theorem (see Theorem \ref{thm: trace}), we can conclude that for $u_h \in \mathcal{U}_h\subset H^{1}(\Omega_h)$,
\begin{equation*}
\gamma^{\rm int}_0u_h\in H^{1/2}(\Gamma_h) \quad \text{and}\quad \|\gamma^{\rm int}_0u_h\|_{H^{1/2}(\Gamma_h)} \le C_{\Omega_h}\|u_h\|_{H^1(\Omega_h)}.
\end{equation*}

In addition to the trace operators $\gamma_0^{\rm int}$ and $\gamma_0^{\rm ext}$, we  define the interior and exterior conormal derivative, for $x\in \Gamma_h$, by
\begin{align*}
\gamma^{\rm int}_1 u(x) &:= \lim_{\Omega_h\ni y \rightarrow x\in \Gamma_h } n(y) \cdot \D u(y),\quad \text{and}\\
\gamma^{\rm ext}_1 u(x) &:= \lim_{\Omega_h^\complement\ni y \rightarrow x\in \Gamma_h } n(y) \cdot \D u(y),
\end{align*}
where $n$ is the outward unit normal vector to $\Omega_h$, i.e. pointing \emph{into} $\Omega_h^\complement$.

Denote the fundamental solution to the Laplace operator in 2D by $G(x,y)$, i.e.
\begin{equation*}
G(x,y): = -\frac{1}{2\pi}\log |x-y|.
\end{equation*}
For $y_0\in \Omega_h$ and $R>2 {\rm diam}(\Omega_h)$, let $B_R(y_0)$ be a ball centred at $y_0$ with radius $R$. Then, by Green's First Identity, we can solve the exterior Laplapce problem \eqref{eq:lap} using the following representation formula, for $x\in B_R(y_0)\setminus \bar{\Omega}_h$,
\begin{equation*}
\begin{aligned}
v(x)  &=  \int_{\Gamma_h}( \gamma^{\rm ext}_0\bar{u}_h)(y)\gamma^{\rm ext}_{1,y}G(x,y)\dd S(y)- \int_{\Gamma_h}G(x,y) \gamma^{\rm ext}_{1}v(y) \dd S(y) + \\
&\quad + \int_{\pp B_R(y_0)} G(x,y)\gamma^{\rm int}_1v(y)\dd S(y) - \int_{\pp B_R(y_0)}\gamma^{\rm ext}_{1,y}G(x,y) \gamma^{\rm int}_0 v(y)\dd S(y).
\end{aligned}
\end{equation*}
Taking limit $R\rightarrow \infty$ gives, for $x \in \Omega_h^\complement$,
\begin{equation}\label{eq:lap_sol}
v(x)  =  u_0 + \int_{\Gamma_h}( \gamma^{\rm ext}_0\bar{u}_h)(y)\gamma^{\rm ext}_{1,y}G(x,y)\dd S(y)- \int_{\Gamma_h}G(x,y) \gamma^{\rm ext}_{1}v(y) \dd S(y),
\end{equation}
where $u_0$ is the far-field constant in \eqref{eq:lap}.

Let us define the following boundary integrals, for $x \in \R^2 \setminus \Gamma_h$,
\begin{equation*}
\begin{aligned}
A\psi(x) &:= \int_{\Gamma_h}  G(x,y) \psi(y)\dd S(y) \quad \text{(single layer potential)},\\
B\psi (x) &:= \int_{\Gamma_h} \psi  (y)  \gamma^{\rm int}_{1,y} G(x,y)\dd S(y)\quad \text{(double layer potential)}. \\
\end{aligned}
\end{equation*}

Then for  $x \in {\Gamma_h}$ we define
\begin{align*}
Vu(x) &:= \gamma^{\rm int}_0 (Au)(x),  \quad Ku(x) := \gamma^{\rm int}_0 (Bu)(x) \\
K'u(x) &:=  \gamma^{\rm int}_1(Au)(x), \quad Du(x) := -\gamma^{\rm int}_1(Bu)(x).
\end{align*}
Applying the exterior trace operator and the exterior conormal operator to \eqref{eq:lap_sol} gives, for $x\in \Gamma_h$,
\begin{align}
\gamma^{\rm int}_0 v(x) &= u_0 + \lambda(x) \gamma^{\rm int}_0\bar{u}_h + (K\gamma^{\rm int}_0\bar{u}_h)(x) - V(\gamma_1^{\rm ext} v)(x), \label{eq: bdry_sys_1}\\
\gamma^{\rm ext}_1 v(x) & = (1-\lambda(x))\gamma_1^{\rm ext}v(x) - (K'\gamma_1^{\rm ext}( v) (x) - (D\gamma_0^{\rm int} \bar{u}_h)(x), \label{eq: bdry_sys_2}
\end{align}
where by Lemma 6.8 in \cite{Olaf:BEM}
\begin{equation*}
\lambda(x) := \lim_{\epsilon\rightarrow 0}\frac{1}{2\pi \epsilon}\int_{y\ni \Omega_h:|y - x| = \epsilon} \dd S(y)  = \frac{1}{2} \text{ a.e}.
\end{equation*}

We observe that the Neumann data $\gamma_1^{\rm ext} v $ can be obtained from the Dirichlet data $\gamma^{\rm int}_0\bar{u}_h$ via \eqref{eq: bdry_sys_1} and \eqref{eq: bdry_sys_2} up to constant $u_0$. To make sure that the operator $V$ is bijective, we need the following restriction on the boundary spaces.

Let us define subspaces
\begin{align*}
H_*^{-1/2}(\Gamma_h) := &\{ w\in H^{-1/2}(\Gamma_h) : \< w,1\>_{\Gamma_h} = 0\} \quad \text{and}\\
H_*^{1/2}(\Gamma_h): = &\{ v\in H^{1/2}(\Gamma_h): v = V(w) \text{ for some } w\in H^{-1/2}_*\}.
\end{align*}
Then Lemma \ref{lem:iso} shows that $V:H_*^{-1/2}(\Gamma_h)\rightarrow H_*^{1/2}(\Gamma_h)$ is an isomorphism and consequently $u_0 = 0$.

\begin{remark}
	For any Lipschitz boundary $\Gamma$, there exist an unique $w_\Gamma \in H^{-1/2}(\Gamma) \setminus H^{-1/2}_*(\Gamma)$ such that
	$\<w_\Gamma, 1\>_\Gamma = 1$ and
	\begin{equation}\label{eq:proj_to_*}
		u - \<u, w_\Gamma\>_{\Gamma} \in H^{1/2}_*(\Gamma),
		\qquad \text{ for any } u\in H^{1/2}(\Gamma).
	\end{equation}
	Its derivation is shown in \cite[\S 6.6.1]{Olaf:BEM}. 
\end{remark}

\medskip

Therefore \eqref{eq: bdry_sys_1} gives
\begin{equation*}
-\gamma_1^{\rm ext} v= V^{-1}(-K+\tfrac{1}{2} I)\gamma^{\rm int}_0\bar{u}_h, \quad \text{if } \gamma^{\rm int}_0\bar{u}_h \in H_*^{1/2}(\Gamma_h).
\end{equation*}
Denote $ g^{-1} := V^{-1}(-K+\tfrac{1}{2}I)$, which is called \emph{Steklov--Poincar\'{e} operator}. Then the total energy (\ref{eq:Etot}) is equivalent to, for $u \in \mathcal{U}_h \cap H_*^{1/2}(\Gamma_h)$,
\begin{equation}\label{eq:E_tot*}
\E^{\rm ac}_* (u) \equiv \E^{\rm tot}(u) : = \E^{\rm ac}_h(u) + \frac{\mu}{2}  \int_{\Gamma_h} (\gamma^{\rm int}_0 u) g^{-1} (\gamma^{\rm int}_0 u).
\end{equation}
Theorem \ref{thm:+ve_def} establishes that Steklov--Poincar\'{e} operator $g^{-1}:  H_*^{1/2}(\Gamma_h)\rightarrow H_*^{-1/2} (\Gamma_h) $ is positive definite. Lemma \ref{lem:g_invariant} shows that $g^{-1}$ is in fact in-variant under rescaling. In addition, in order to ensure that the regularity constants are independent of the size of the boundary $\Gamma_h$, we employ a rescaling argument in Section \ref{sect:rescale} to introduce another fractional norm on the boundary: for $u\in H^{1/2}(\Gamma_h)$
\begin{equation*}
\|u\|^2_{H^{1/2}_{\Gamma_h}} : = \big[\tfrac{1}{2}{\rm diam}(\Gamma_h)\big]^{-1}\|u\|^{2}_{L^2(\Gamma_h)} +|u|^2_{H^{1/2}(\Gamma_h)}.
\end{equation*}
By Lemma \ref{lem:g-1_rescale} we have that for all $u\in H^{1/2}_*(\Gamma_h)$
\begin{equation*}
\<g^{-1}u, u\>  \ge C_1 \|u\|^2_{ H^{1/2}_{\Gamma_h}} \quad \text{and}\quad  \|g^{-1}u\|_{H^{-1/2}_{\Gamma_h}} \le C_2\|u\|_{H^{-1/2}_{\Gamma_h}},
\end{equation*}
where $C_1$ and $C_2$ are independent of the radius of $\Omega_h$.

Now we take in account of the displacement inside $\Omega_h$ to introduce the following norm for the error analysis. For $u \in \mathcal{U}_h \cap H_*^{1/2}(\Gamma_h)$, define
\begin{equation}\label{eq:def_complete_norm}
\|u\|^2_E : = \|\D u\|^2_{L^2(\Omega_h)} + \|u\|^2_{H^{1/2}_{\Gamma_h}}.
\end{equation}
It is clear that this norm is rescale in-variant.

\subsubsection{Boundary element method}
\label{sec:bem}
We introduce a numerical discretization scheme to approximate the boundary integral equations. Let
\begin{equation*}
S^0_h(\Gamma_h) = \text{\rm span}\big\{\phi_k^0\big\}_{k = 1}^M \subset H^{-1/2}_*(\Gamma_h),
\end{equation*}
where $\phi_k^0$ are piecewise constant basis functions on the discretized boundary with elements $\mathcal{T}_h \cap \Gamma_h$.
For a Dirichlet data $u \in H^{1/2}_*(\Gamma)$, we define $g_h^{-1}u: = \bar{v}_h\in S^0_h(\Gamma_h)$ as the solution to
\begin{equation}\label{eq: def_g_h}
\<V \bar{v}_h, \tau_h\> = \<(K+(\lambda -1)I)u,\tau_h\>, \quad \text{for all }\tau_h\in S^0_h(\Gamma_h).
\end{equation}
Then we define
\begin{equation}\label{eq:Etot_h}
\E ^{\rm tot}_h(u_h): = \E^{\rm ac}_h(u_h) + \frac{\mu}{2}  \int_{\Gamma_h}  (\gamma^{\rm int}_0u_h) g_h^{-1}(\gamma^{\rm int}_0 u_h),
\end{equation}
where $\gamma^{\rm int}_0u_h \in S^0_h(\Gamma_h) $.
We seek the solution to
\begin{equation}\label{eq: solution_uh}
u_h : = \arg \min \{\E^{\rm tot}_h (u) - f(u)\,:\, u \in \U_h^*\},
\end{equation}
where 	\begin{equation*}
\mathcal{U}_h^*  := \big\{u_h \in \mathcal{U}_h \cap S^0_h(\Gamma_h)\, : \, u_h |_{\Gamma_h}\in H_*^{1/2}(\Gamma_h)\big\},
\end{equation*}
and the error estimate $\|\D u_h- \D \tilde{u}^\a \|$ in a suitable norm.

For the simplicity of analysis, we impose the following assumption on the
boundary $\Gamma_h$ and the atomistic triangulation $\T$:

\medskip

{\bf (A3) }  The boundary $\Gamma_h$ is aligned with the canonical triangulation $\T$ in the sense that, for all $T \in \T $,
\begin{enumerate}
	\item[(a)] $T \cap \Gamma_h \neq \emptyset \implies {\rm int} (T) \cap \Gamma_h = \emptyset$.
	\item[(b)] Let $\mathcal{V}_{\rm FEM}$ be the set of vertices of $\T _h$, and $\mathcal{V}_{\rm can}$ be the set of vertices of $\T$, then
	$\mathcal{V}_{\rm FEM } \cap \Gamma_h \subset \mathcal{V}_{\rm can}$.
\end{enumerate}

\medskip \noindent {\bf (A3)} is  employed in
\S~\ref{sec: test_fun_v} for the construction of a dual interpolant. We expect
that, without it, the main results are still true, but would require some
additional technicalities to prove. For the sake of clarity we therefore impose
{\bf (A3)} to emphasize the main concepts of the error analysis.

\section{Preliminaries}\label{sect:prelim}

%
%
%

In order to measure the ``smoothness'' of displacement maps $u \in \dot{\U }^{1,2}$, we review from \cite{2014-bqce} a smooth interpolant $\tilde{u} $,  namely a
$C^{2,1}$-conforming multi-quintic interpolant.

\begin{lemma}\label{lem:smooth_int}
	(a) For each $u :\L \rightarrow \R ^m$, there exists a unique $\tilu \in C^{2,1}(\R^2; \R^m)$ such that, for all $\ell \in \L$,
	\begin{equation*}
	\begin{aligned}
	\left.\tilu\right|_{\ell + \mA(0,1)^2} &\text{ is a polynomial of degree 5}, \\
	\tilu (\ell) &= u(\ell),\\
	\partial_{a_i} \tilu (\ell) &= \tfrac{1}{2}\left(u(\ell + a_i) - u(\ell - a_i)\right),\\
	\partial_{a_i}^2 \tilu (\ell)& = u(\ell+a_i) - 2u(\ell) +u(\ell- a_i),
	\end{aligned}
	\end{equation*}
	where $i \in \{1,2\}$ and $\partial_{a_i}$ is the derivative in the
	direction of $a_i$.

	(b) Moreover, for $q \in [1,\infty]$, $0\le j \le 3$,
	\begin{equation}\label{eq: int_ineq_proof}
	\|\D^j \tilu \|_{L^q(\ell+\mA(1,0)^2)} \lesssim \|D^j u \|_{\ell^q\left(\ell + \mA\{-1,0,1,2\}^2 \right)} \quad \text{and } \quad  |D^j u(\ell)| \lesssim\|\D^j \tilu \|_{L^1(\ell+\mA(-1,1)^2)},
	\end{equation}
	where $D$ is the difference operator defined in \eqref{def:diff_op}.  In
	particular,
	\begin{equation*}\|\D \tilde{u}\|_{L^q} \lesssim \|\D u\|_{L^q} \lesssim \|\D \tilde{u}\|_{L^q},
	\end{equation*}
	where $ u $ is identified with its piecewise affine interpolant.
\end{lemma}
\begin{proof}
	This is the same proof as Lemma 6.1 in \cite{2016-qcp2}.
\end{proof}

\subsection{Properties of Steklov--Poincar\'{e} operator}\label{sect: steklov-poincare}
As mentioned in Section \ref{sect:model}, we require some regularity properties of the Steklov--Poincar\'{e} operator $g^{-1}$. First of all we have the following trace theorem.

\begin{theorem}[Trace Theorem]\label{thm: trace}
	For $\frac{1}{2}<s\le1$,  the interior trace operator
	\begin{equation*}
	\gamma_0 :H^{s}(\Omega_h) \rightarrow H^{s-1/2}(\Gamma_h)
	\end{equation*}
	is bounded satisfying
	\begin{equation*}
	\|\gamma_0 v\|_{H^{s-1/2}(\Gamma_h)}\le c_T\|v\|_{H^s(\Omega_h)}, \quad \forall v\in H^s(\Omega_h).
	\end{equation*}
\end{theorem}
\begin{proof}
	This is a standard result, see for example \cite{Adam:soblev}.
\end{proof}

The boundedness and ellipticity of the boundary integrals are proved in \cite{Costabel} for Lipschitz domains.

\begin{theorem}[Boundedness] \label{thm:bdry_bdd}
	The boundary integral operators
	\begin{align*}
	V :& H^{-1/2+s}(\Gamma_h) \rightarrow H^{1/2+s}(\Gamma_h), \\
	K :& H^{1/2+s}(\Gamma_h) \rightarrow H^{1/2+s}(\Gamma_h), \\
	K': &	H^{1/2+s}(\Gamma_h) \rightarrow H^{1/2+s}(\Gamma_h),\\
	D :& H^{1/2+s}(\Gamma_h) \rightarrow H^{-1/2+s}(\Gamma_h)
	\end{align*}
	are bounded for all $s\in \left(-\frac{1}{2},\frac{1}{2}\right)$.
	%
	%
	%
\end{theorem}
\begin{proof}
	See Theorem 1 in \cite{Costabel}.
\end{proof}

\begin{theorem}[Ellipticity] \label{thm: ellip}The operators $V$ and $D$ are strongly elliptic in the sense that, there exists $C^V, C^D >0$ such that for all $v\in H^{-1/2}_*(\Gamma_h), u \in H^{1/2}_*(\Gamma_h)$
	\begin{align}
	\<V v, v\> &\ge C^V \|v\|^2_{H^{-1/2}(\Gamma_h)},\\
	\<D u,u\> &\ge C^D \|u\|^2_{H^{1/2}(\Gamma_h)}.
	\end{align}
\end{theorem}
\begin{proof}
	This is a special case of Theorem 2 in \cite{Costabel}. In 2D, the far-field constant $u_0$ vanishes only if the Dirichlet data is in the subspace $H_*^{1/2}(\Gamma_h)$. See Appendix \ref{app:ellip} for a full proof.
\end{proof}

\begin{lemma}\label{lem:iso}
	$V:  H^{-1/2}_*(\Gamma_h) \rightarrow  H^{1/2}_*(\Gamma_h)$ is an isomorphism.
\end{lemma}
\begin{proof}
	This is a consequence of Theorems \ref{thm:bdry_bdd} and \ref{thm: ellip} and the Lax-Milgram Theorem.
\end{proof}
%
Therefore, with the boundedness and ellipticity, we can prove the positive definiteness of the Steklov--Poincar\'{e} operator.

\begin{theorem}\label{thm:+ve_def}	The Steklov--Poincar\'{e} operator $g^{-1}: H^{1/2}_*(\Gamma_h)\rightarrow  H^{-1/2}_*(\Gamma_h)$ is well-defined. Furthermore,  there exist $C^g_1, C^g_2>0$ such that for all $u \in H^{1/2}_*(\Gamma_h) $
	\begin{equation}
	\<g^{-1}u, u\>  \ge C^g_1 \|u\|^2_{ H^{1/2}(\Gamma_h)} \quad \text{and}\quad  \|g^{-1}u\|_{H^{-1/2}(\Gamma_h)} \le C^g_2\|u\|_{H^{1/2}(\Gamma_h)}.
	\end{equation}
\end{theorem}
\begin{proof} Since $V: H^{-1/2}_*(\Gamma_h)\rightarrow H^{1/2}_*(\Gamma_h)$ is an isomorphism and $K:H^{1/2}(\Gamma_h)\rightarrow H^{1/2}(\Gamma_h)$ is bounded, we have $g^{-1} = V^{-1}(-K+(1-\lambda)I)$ is well-defined. The upper bound $C^g_2$ follows from the Lax-Milgram Theorem.

	For positive-definiteness, we use an analogous argument to that in \cite{Olaf:BEM}. We first observe that for any $u \in H^{1/2}_*(\Gamma_h)$, there exists an unique solution $v$ to the Laplace problem
	\begin{equation*}
	\begin{aligned}
	-\Delta v &= 0, \quad \text{in } \Omega_h^{\complement}\\
	v & = u, \quad \text{on } {\Gamma_h} \\
	|v(x)| & = \mathcal{O}\left({\frac{1}{|x|}}\right) \quad \text{as }|x|\rightarrow \infty
	\end{aligned}
	\end{equation*}
	with $g^{-1}u = -\gamma^{\rm int}v$.
	Similar to \eqref{eq: bdry_sys_1} and \eqref{eq: bdry_sys_2}, we have the  relationships
	\begin{align*}
	v(x) &= \lambda(x) u + (Ku)(x) - V(\gamma_1^{\rm ext} v)(x)\quad \text{and} \\
	\gamma^{\rm ext}_1 v(x) & = (1-\lambda(x))\gamma_1^{\rm ext}v(x) - (K'\gamma_1^{\rm ext}( v) (x) - Du(x).
	\end{align*}
	Combining these two equations we obtain an alternative representation for $g^{-1}$:
	\begin{equation*}
	g^{-1}(u)= 	-\gamma_1^{\rm ext}v  = Du + (K'-(1-\lambda)I) V^{-1}(K - (1-\lambda)I) (u).
	\end{equation*}
	Consequently we have
	\begin{align*}
	\<g^{-1} u, u\> &= \<D u, u\> + \<V^{-1}(K - (1-\lambda)I)u, (K - (1-\lambda)I)u\>\\
	&\ge \<D u, u\>\\
	&\ge C^D \|u\|^2_{H^{1/2}(\Gamma_h)}. \qedhere
	\end{align*}

\end{proof}

\subsection{Re-scaling of the boundary integrals}\label{sect:rescale}
In the analysis of a/c coupling methods, we are concerned with the convergence rate against the size of the domain. Therefore, we need to explore how boundary integrals scale with the size of the domain.

Suppose that $f_1 : \Gamma_1:= \pp \Omega_1\rightarrow \R ^2$ and $f_1\in H^{1/2}(\pp B_1)$, where $\Omega_1$ is a Lipschitz domain with radius 1. Let $f_R: \Gamma_R: = \pp \Omega_R \rightarrow \R ^2$ and $f_R(x): = f_1(x/R)$, where $ \Omega_R = R\Omega_1$. Then we have
\begin{equation*}
\|f_R\|^2_{L^2(\Gamma_R)} = \int_{\Gamma_R}|f_R(x)|^2\dx = \int_{\Gamma_R} |f_1(x/R)|^2\dx = \int_{\Gamma_1}|f_1(y)|^2 R \dy = R\|f_1\|^2_{L^2(\Gamma_1)},
\end{equation*}
while
\begin{align*}
		|f_R|^2_{H^{1/2}(\Gamma_R)}& := \int_{\Gamma_R}\int_{\Gamma_R} \frac{[f_R(x)-f_R(y)]^2}{|x-y|^{2}}{\rm d} S(x) {\rm d} S(y)\\
			& = \int_{\Gamma_R}\int_{\Gamma_R}\frac{[f_1(x/R)-f_1(y/R)]^2}{|x-y|^{2}}{\rm d} S(x) {\rm d} S(y)\\
			& = \int_{\Gamma_1}\int_{\Gamma_1}\frac{[f_1(x')-f_1(y')]^2}{|Rx'-Ry'|^{2}}R^2{\rm d} S(x') {\rm d} S(y')\\
			& =  \int_{\Gamma_1}\int_{\Gamma_1}\frac{[f_1(x')-f_1(y')]^2}{|x'-y'|^{2}}{\rm d} S(x') {\rm d} S(y')\\
			& = |f_1|^2_{H^{1/2}(\Gamma_1)}.
\end{align*}
Thus we define a re-scaled norm in $H^{1/2}(\Gamma_h)$,
\begin{equation}\label{def:rescale_norm}
	\|f\|^2_{H^{1/2}_{\Gamma_h}}
		:=  \big[\tfrac12 {\rm diam}(\Gamma_h)\big]^{-1} \|f\|^{2}_{L^2(\Gamma_h)} +|f|^2_{H^{1/2}(\Gamma_h)},
\end{equation}
then we have $\|f_1\|_{H^{1/2}(\Gamma_1)} = \|f_R\|_{H^{1/2}_{\Gamma_R}}$
with $R = \frac12 {\rm diam}(\Gamma_h)$.
Similarly, we define a rescaled $H^1$ norm
\begin{equation}\label{def:rescale_norm_h1}
\|f\|^2_{H^{1}_{\Gamma_h}}
	:= \big[\tfrac12 {\rm diam}(\Gamma_h)\big]^{-1} \|f\|^{2}_{L^2(\Gamma_h)}
		+ \big[\tfrac12 {\rm diam}(\Gamma_h)\big] \|\D f\|^2_{L^{2}(\Gamma_h)},
\end{equation}
then we have $\|f_1\|_{H^{1}(\Gamma_1)} = \|f_R\|_{H^{1}_{\Gamma_h}}$
with $R = \frac12 {\rm diam}(\Gamma_h)$.

\begin{lemma} \label{lem:g_invariant} Let $V_1$, $K_1$, $V_R$, $K_R$ be the boundary integrals $V$ and $K$ on $\Gamma_1$ and $\Gamma_R$ respectively. Denote
	\begin{equation*}
	g_1^{-1} :=  V_1^{-1}(-K_1+(1 - \lambda)I)\quad \text{and} \quad  g_R^{-1} :=  V_R^{-1}(-K_R+(1 - \lambda)I).
	\end{equation*}
Then for $u_1 \in H^{1/2}_*(\Gamma_1)$ and $u_R := u_1(x/R)$, we have $u_R \in H^{1/2}_*(\Gamma_R)$ and
\begin{equation*}
\<g^{-1}_1 u_1, u_1\>_{\Gamma_1} = \<g^{-1}_R u_R, u_R\>_{\Gamma_R}.
\end{equation*}
\end{lemma}
\begin{proof}
	See Appendix \ref{app:proof_lem_rescale}.
\end{proof}

%
%
%
Using the re-scaled norm we have the following Lemma.

\begin{lemma}\label{lem:g-1_rescale}
	The Steklov--Poincar\'{e} operator $g^{-1}: H^{1/2}_*(\Gamma_h)\rightarrow  H^{-1/2}_*(\Gamma_h)$ has the following regularity, for $u \in H^{1/2}_*(\Gamma_h) $
	\begin{equation}\label{eq:property_g_rescale}
	\<g^{-1}u, u\>  \ge C_1 \|u\|^2_{ H^{1/2}_{\Gamma_h}} \quad \text{and}\quad  \|g^{-1}u\|_{H^{-1/2}(\Gamma_h)} \le C_2\|u\|_{H^{1/2}_{\Gamma_h}},
	\end{equation}
	where $C_1$ and $C_2$ are independent of the radius of $\Omega_h$.
\end{lemma}
\begin{proof}
	The result follows directly from Theorem \ref{thm:+ve_def} and Lemma \ref{lem:g_invariant}.
\end{proof}

\subsection{Boundary element approximation error}
We also need the following boundary element approximation error estimate to compare $g^{-1}$ and $g^{-1}_h$.

\begin{theorem}\label{thm:g-g_h}
	If $u \in H^{1/2}_*(\Gamma_h) $, then the approximation solution $g_h^{-1} u$ to \eqref{eq: def_g_h} exists and we have the following stability property
	\begin{equation}\label{eq: g_h_stability}
	\|g_h^{-1} u\|_{H^{-1/2}(\Gamma_h)} \le \frac{C_2}{C_1}\|u\|_{H^{1/2}_{\Gamma_h}}.
	\end{equation}
	Furthermore, if $g^{-1}u\in H^{1}(\Gamma_h)$, then
	\begin{equation}\label{eq: g-g_h}
	\|g^{-1}u - g_h^{-1} u\|_{H^{-1/2}(\Gamma_h)}\le Ch^{3/2}|g^{-1}u|_{H^{1}(\Gamma_h)},
	\end{equation}
	where $h$ is the size of each boundary element and $C$ is independent of the size of $\Gamma_h$.
\end{theorem}
\begin{proof}
	Since $S^0_h(\Gamma_h)$ is a conforming trial space in $H^{-1/2}(\Gamma_h)$ and $g^{-1}$ is bounded and elliptic according to Lemma \ref{lem:g-1_rescale}, then by the Lax-Milgram Theorem $g^{-1}_h u$ exists and we have \eqref{eq: g_h_stability}. For \eqref{eq: g-g_h}, by Cea's Lemma we have
	\begin{equation*}
	\|g^{-1}u - g_h^{-1} u\|_{H^{-1/2}(\Gamma_h)}\le \frac{C_2}{C_1}\inf_{v_h\in H^{-1/2}(\Gamma_h)} \| g^{-1}u - v_h\|_{H^{-1/2}(\Gamma_h)}.
	\end{equation*}
	By Theorem 10.4 in \cite{Olaf:BEM}, we have
	\begin{equation*}
	\inf_{v_h\in H^{-1/2}(\Gamma_h)} \| g^{-1}u - v_h\|_{H^{-1/2}(\Gamma_h)} \le c h^{3/2}|g^{-1}u|_{H^{1}(\Gamma_h)}.
	\end{equation*}
\end{proof}

\section{Main results}\label{sec:main_results}

\subsection{Regularity of $u^\a$}
The approximation error estimates in later sections requires the
decay of the elastic fields away from the defect core which follows from
a natural stability assumption:

\bf{(A1)} \normalfont The atomistic solution is strongly stable, that is, there
exists $C_0>0$,
\begin{equation}\label{assum:stab_hom}
\<\del^2 \E^{\a} (u^\a) \varphi, \varphi  \> \ge C_0 \|\D \varphi\|^2_{L^2},
\quad \forall \varphi \in  \dot{\mathcal{U}}^{1,2},
\end{equation}
where $u^\a$ is a solution to \eqref{eq:y_a}.

\begin{corollary}\label{thm:decay}
	Suppose that {\bf (A1)} is satisfied, then there exists a constant $C>0$
	such that, for $1\le j \le 3$,
	\begin{equation*}
	|D^ju^\a(\ell)|\le
	C| \ell |^{-1-j} \quad \text{and} \quad |\D^j \tilde{u}^\a (x)| \le C|x|^{-1-j}.
	\end{equation*}
	
\end{corollary}

\begin{proof}
	See Theorem 2.3 in \cite{EhrOrtSha:2013}.
\end{proof}

\subsection{Stability} \label{sec:stability}
In \cite{2013-stab.ac} it is proven that there is a ``universal'' instability in
2D interfaces for QNL-type a/c couplings. It is impossible to show
that $\ddel \E^{\rm g23}(u^\a)$ is a positive definite operator for general cases, even
with the assumption \eqref{assum:stab_hom}. In fact, this potential instability is
universal to a wide class of generalized geometric reconstruction
methods. Nevertheless, it is rarely observed in practice. To circumvent this
difficulty, we make the following standing assumption:

\bf{(A2)} \normalfont The {\em homogeneous lattice} is strongly stable under the
G23 approximation, that is, there exists $C^{\rm ac}_0 > 0$ which is independent of $K$
such that, for $K$ sufficiently large,
\begin{equation}\label{assum:stab_hom_g23}
\<\del^2 \E_h^{\rm ac}(0)\varphi_h, \varphi_h  \> \ge C^{\rm ac}_0\|\D \varphi_h\|^2_{L^2}, \quad \forall \varphi_h \in  \mathcal{U}_h.
\end{equation}

Because \eqref{assum:stab_hom_g23} does not depend on the solution it can be
tested numerically. But a precise understanding under which conditions
\eqref{assum:stab_hom_g23} is satisfied is still missing. In \cite{2013-stab.ac}
a method of stabilizing 2D QNL-type schemes with flat interfaces is formulated,
which could replace this assumption, but we are not yet able to extend this method to interfaces with corners, such as the configurations
discussed in this paper. From these two assumptions, we can deduce the following stability result when the BEM formulation is added.

\begin{lemma}\label{thm:stab_g23_0}
	For any $\varphi_h\in \mathcal{U}_h^*$, we have
	\begin{equation}
	\<\del^2 \E^{\rm tot}_h(0)\varphi_h, \varphi_h  \> \ge C^{\rm tot}_0\| \varphi_h\|^2_{E},
	\end{equation}
	where $\| \cdot \|^2_E$ is the norm defined in \eqref{eq:def_complete_norm} and $C_0^{\rm tot}$ is independent of the size of $\Omega_h$.
\end{lemma}
\begin{proof}
	This is an immediate consequence of the property \eqref{eq:property_g_rescale} of $g^{-1}$:
	\begin{align*}
	\<\del^2 \E^{\rm tot}_h(0)\varphi_h, \varphi_h  \>  &= \<\del^2 \E^{\rm ac}_h(0)\varphi_h, \varphi_h  \> + \mu \int_{\Gamma_h}\varphi_hg^{-1} \varphi_h\\
	&\ge C^{\rm ac}\|\D \varphi_h\|^2_{L^2(\Omega_h)} + C_1 \|\varphi_h\|^2_{ H^{1/2}_{\Gamma_h}}\\
	&\ge \min \{C^{\rm ac}, C_1\}\|\varphi_h\|^2_E. \qedhere
	\end{align*}
\end{proof}

Then we have the following stability estimate.

\begin{theorem}\label{thm:stab}
	Under assumptions \textbf{(A1)} and \textbf{(A2)} there exists $\gamma>0$ such that, when the atomistic region radius $K$ is sufficiently large,
	\begin{equation}\label{eq:stab}
	\langle \delta \mathcal{G}_h(\Pi_h u^\a)\varphi_h,\varphi_h\rangle \ge \gamma \| \varphi_h\|^2_{E} \quad
	\text{ for all }\varphi_h \in \mathcal{U}^*_h.
	\end{equation}
\end{theorem}
\begin{proof}
	After employing Lemma~\ref{thm:stab_g23_0} this is a straightforward
	adaptation of the proof of \cite[Lemma 4.9]{2014-bqce}.
\end{proof}

\subsection{Main results}
Our two main results are a consistency error estimate for the A/C+BEM coupling
scheme and the resulting error estimate.

\begin{theorem}[Consistency]
	\label{thm:cons_main}If $u^\a $ is a solution to \eqref{eq:y_a}, then for all $v_h \in \mathcal{U}_h^*$
	\begin{equation}\label{eq:main_result_cons}
	\begin{aligned}
	&\hspace{-1.5cm} \<\del \E_h ^{\rm tot }(\Pi_h u^\a), v_h\> \\
	\lesssim &\Big(\|\D ^2\tilu ^\a \|_{L^2(\Omega^\i) } +\|\D ^3\tilu ^\a \|_{L^2(\R ^2 \setminus \Omega ^\a )} +\|\D ^2\tilu ^\a \|^2_{L^4(\R ^2 \setminus \Omega ^\a )} \\
	&\quad + \|h\D ^2\tilu ^\a \|_{L^2(\Omega^\c _h)} + \|h^{3/2}\D ^2\tilu ^\a \|_{L^2(\Gamma_h)} + N^{-3}\Big) \|v_h\|_{E}.
	\end{aligned}
	\end{equation}
\end{theorem}
\begin{proof}
	See Section \ref{sect:proof_thm_consit}.
\end{proof}

Combining Theorem \ref{thm:cons_main} with the stability result Theorem \ref{thm:stab}, we obtain the following error estimate.

\begin{theorem}\label{thm:err_2}
If $u^\a $ is a solution to \eqref{eq:y_a} and Assumptions (A1) and (A2) are satisfied then, for $K$ sufficiently large, there exists a solution $u_h\in \mathcal{U}_h^*$ to \eqref{eq: solution_uh} satisfying
\begin{equation}
\label{eq: final_err}
\begin{aligned}
\| \tilu ^\a - u_h\|_{E} &\lesssim \|\D ^2\tilu ^\a \|_{L^2(\Omega^\i) } +\|\D ^3\tilu ^\a \|_{L^2(\R ^2 \setminus \Omega ^\a )} +\|\D ^2\tilu ^\a \|^2_{L^4(\R ^2 \setminus \Omega ^\a )} \\
	& \quad + \|h\D ^2\tilu ^\a \|_{L^2(\Omega^\c _h)}
 		+ \| h^{3/2} \D ^2\tilu ^\a \|_{L^2(\Gamma_h)} + N^{-3}.
\end{aligned}
\end{equation}
\end{theorem}

\begin{proof}
	See Section \ref{sect:proof_thm_err_2}.
\end{proof}

\begin{remark}
	The term $N^{-3}$ is in fact the linearization error. Recall that in \eqref{eq:E_lin} we approximate the Cauchy--Born strain energy $W(\D u)$ by the linearised elasticity strain energy $\tfrac{1}{2}\mu |\D u|^2$. The linearization error in first variation can (formally) be estimated by
	\begin{equation*}
	\int_{\Omega_h^\complement} \left[\pp_\mF W(\D u)\D v - \mu \D u\cdot \D v\right]
	\lesssim \int_{\Omega_h^\complement} |\D u|^2 |\D v| \lesssim \|\D u\|^2_{L^4(\Omega^\complement_h)}\|\D v\|_{L^2(\Omega_h^\complement)}.
	\end{equation*}
	Taking account of the decay of $\tilu ^\a $ from Corollary \ref{thm:decay}, we have
	\begin{equation*}
	\|\D \tilu^\a \|^2_{L^4(\Omega^\complement_h)} \lesssim N^{-3}.
	\end{equation*}
	For technical reasons we cannot directly perform such an estimate, but
	the $O(N^{-3})$ term arises in an indirect way; cf. \S \ref{sec:A3A4} and \ref{sec:A5}.
\end{remark}

\subsection{Optimal approximation parameters}\label{sec:opt_param}
In \cite{2016-qcp2} we discussed the optimization of mesh parameters for P1-FEM and P2-FEM. We now perform a similar analysis for the setting of the present work,
including the BEM approximation of the elastic far-field.


Recall that $K$ is the radius of atomistic region $\Omega^\a $ and $N$ is the radius of $\Omega_h$. To simplify the discussion we assume that the FE mesh grading is linear,
$|h(x)| \approx |x| / K$, which unsures quasi-optimal computational cost,
up to logarithmic terms. In this setting it is easy to see that the
various error contributions in \eqref{eq: final_err} are bounded by
\begin{equation*}
\begin{aligned}
\text{Modelling error:} \quad &\|\D ^2\tilu ^\a \|_{L^2(\Omega^\i) } +\|\D ^3\tilu ^\a \|_{L^2(\R ^2 \setminus \Omega ^\a )} +\|\D ^2\tilu ^\a \|^2_{L^4(\R ^2 \setminus \Omega ^\a )}  \lesssim K^{-5/2},\\
\text{FEM error: } \quad &\|h\D ^2\tilu ^\a \|_{L^2(\Omega^\c _h)} \lesssim \big(K^{-4} - K^{-2} N^{-2}\big)^{1/2},\\
\text{BEM error: } \quad &\| h^{3/2} \D ^2\tilu ^\a \|_{L^2(\Gamma_h)} \lesssim K^{-3/2} N^{-1}, \qquad \text{and} \\
\text{Linearisation error: } \quad & N^{-3}.
 \end{aligned}
\end{equation*}

The key observation is that the modelling error, which cannot be reduced by
choice of $N$ or $h$ is $O(K^{-5/2})$. By choosing $N \leq K + C$ for some
fixed constant, both the
FEM and the BEM errors also become $O(K^{-5/2})$, whereas for $N \gg K$,
we obtain that the FEM error contribution becomes $O(K^{-2})$ which is strictly
larger.

This quasi-optimal balance of approximation parameters means that we ought
to remove the nonlinear elasticity region and directly couple the atomistic
model to the BEM. The resulting error estimate is
\begin{equation}\label{eq: err_N_K}
	\| \tilu ^\a - u_h\|_{E} \lesssim K^{-5/2},
\end{equation}
which is the best possible rate that can be achieved for a sharp-interface
coupling method.

We remark, however, that the interface region (and therefore a thin layer
of Cauchy--Born elasticity) cannot be removed entirely since the BEM must
be coupled to a local elasticity model (FEM) rather than directly to the
atomistic model. Coupling directly to the atomistic model would lead to
a new consistency error usually dubbed ``ghost forces''.


\section{Conclusion}
In this work we have explored the natural combination of atomistic, finite
element and boundary element modelling from the perspective of error analysis.
The conclusion is an interesting, albeit not entirely
unexpected, one. The rapid decay of elastic fields in the point defect case
$|\D^j \tilu(x)| \lesssim |x|^{-1-j}$ means that the continuum model error
$|\D^3 \tilu|$ and
and linearisation error $|\D\tilu|^2$ are balanced. It is therefore reasonble
to entirely bypass the nonlinear elasticity model and couple the atomistic
region directly to a linearised elasticity model. This observation, as well as
additional complexities due to finite element and boundary element
discretisation errors are made precise in Theorem~\ref{thm:err_2} and in the
discussion in \S~\ref{sec:opt_param}.

Because the characteristic decay of elastic fields is different for
different material defects (or other materials modelling situations)
our conclusion cannot immediately applied to other contexts. However in
those sitations our analysis can still provide guidance on how to generalise
our results and optimally balance approximation errors due to
continuum approximations, linearisation, finite element and boundary
element approximations.


\section{Proofs: Reduction to consistency}
\label{sec:reduction to consistency}

Assuming the existence of an atomistic solution $u^\a$ to \eqref{eq:y_a}, we seek to prove the
existence of $u_h^{\rm ac} \in \mathcal{U}_h^*$ satisfying
\begin{equation}\label{dE_h}
\< \del \E^{\rm tot}_h (u_h^{\rm ac}), \varphi_h\>
= \<\delta f(u_h), \varphi_h\>,
\quad \text{for all }\varphi_h \in \mathcal{U}_h^*,
\end{equation}
and to estimate the error $\|u^\a - u_h^{\rm ac}\|_{E}$.

The error analysis consists of a best-approximation analysis
(\S~\ref{sect:best_approx}), consistency and stability estimates
(\S~\ref{sec:prf-stab}). Once these are established we apply a formulation of
the inverse function theorem (\S~\ref{sec:ift}) to obtain the existence of a
solution $u_h^{\rm ac}$ and the error estimate.

\subsection{The best approximation operator}\label{sect:best_approx}
We define a quasi-best approximation map $\Pi_h: C(\R^2; \R^m) \rightarrow \mathcal{U}_h$ to be the
nodal interpolation operator, i.e., for $f \in C(\R^2; \R^m)$,
$\Pi_h(f)|_T \in \mathbb{P}^1(T)$ for
$ T\subset \mathcal{T}_h$ and
\begin{displaymath}
\Pi_h(f)(x) =  f(x) - f_0 \qquad \text{ for all $x \in \mathcal{N}_h$},
\end{displaymath}
where $f_0$ is a constant such that $f(x) - f_0\in H_*^{1/2}(\Gamma_h)  $ for $x\in \Gamma_h$.
Then it is clear that $\Pi_hu^\a \in \U_h^*$.

\subsection{Inverse Function Theorem}
\label{sec:ift}
The proof of
this theorem is standard and can be found in various references,
e.g. \cite[Lemma 2.2]{Ortner:qnl.1d}.

\begin{theorem}[The inverse function theorem]\label{theo:inverse}
	Let $\mathcal{U}_h$ be a subspace of $\mathcal{U}$, equipped with
	$\|\nabla \cdot \|_{L^2}$, and let
	$\mathcal{G}_h \in C^1(\mathcal{U}_h,\mathcal{U}_h^*)$ with
	Lipschitz-continuous derivative $\delta \mathcal{G}_h$:
	\begin{equation*}
	\|\delta \mathcal{G}_h(u_h)-\delta \mathcal{G}_h(v_h)\|_\mathcal{L} \le M \|\nabla u_h-\nabla v_h\|_{L^2} \quad \text{for all } u_h,v_h \in \mathcal{U}_h,
	\end{equation*}
	where $\|\cdot \|_\mathcal{L}$ denotes the
	$\mathcal{L}(\mathcal{U}_h,\mathcal{U}_h^*)$-operator norm.

	Let $\bar{u}_h \in \mathcal{U}_h$ satisfy
	\begin{align}
	\|\mathcal{G}_h(\bar{u}_h)\|_{\mathcal{U}_h ^*} &\le \eta, \label{inv_consist} \\
	\langle \delta \mathcal{G}_h(\bar{u}_h)v_h,v_h\rangle &\ge \gamma \|\nabla v_h\|^2_{L^2} \quad
	\text{ for all }v_h \in \mathcal{U}_h, \label{inv_stab}
	\end{align}
	such that $M,\eta, \gamma$ satisfy the relation
	\begin{equation*}
	\frac{2M\eta}{\gamma^2}<1.
	\end{equation*}
	Then there exists a (locally unique) $u_h \in \mathcal{U}_h$ such that $\mathcal{G}_h(u_h) = 0$,
	\begin{align*}
	\|\nabla u_h-\nabla \bar{u}_h\|_{L^2}&\le 2\frac{\eta}{\gamma}, \quad \text{and} \\
	\langle \delta \mathcal{G}_h(u_h)v_h,v_h\rangle & \ge \left( 1- \frac{2M\eta}{\gamma^2}\right)\gamma \|\nabla v_h\|^2_{L^2} \quad \text{for all } v_h \in \mathcal{U}_h.
	\end{align*}
\end{theorem}

To put Theorem \ref{theo:inverse} (Inverse Function Theorem) into our context,
let
\begin{equation*}
\mathcal{G}_h(v) : =  \del \E^{\rm tot}_h(v) -\del f(v) \quad \text{and}\quad \bar{u}_h := \Pi_h u^\a,
\end{equation*}
where $u^\a$ is a solution to \eqref{eq:y_a}.

To make \eqref{inv_consist} and \eqref{inv_stab} concrete we will show that
there exist $\eta,\gamma >0$ such that, for all $\varphi_h \in \mathcal{U}^*_h$,
\begin{equation*}
\begin{aligned}
\langle \del \E^{\rm tot}(\Pi_h u^\a ),\varphi_h\rangle - \<\del f(\Pi_h u^\a), \varphi_h\>&\le \eta \| \varphi_h\|_{ E}, \quad (consistency) \\
\langle \del^2 \E^{\rm tot}(\Pi_h u^\a ) \varphi_h, \varphi_h \rangle -\<\del^2 f(\Pi_h u^\a)\varphi_h, \varphi_h\> &\ge \gamma  \| \varphi_h\|^2_{ E}. \quad (stability)
\end{aligned}
\end{equation*}
Ignoring some technical requirements, the inverse function theorem implies that, if $\eta / \gamma $ is sufficiently small, then there exists $u^{\rm ac}_h \in \mathcal{U}_h^*$ such that
\begin{equation*}
\begin{aligned}
&\langle \del \E^{\rm tot}(u_h^{\rm ac} ),\varphi_h\rangle - \< \del f(u_h^{\rm ac}), \varphi_h\>=0, \quad \forall \varphi_h \in \mathcal{U}^*_h, \quad \text{and}\\
&\| u^{\rm ac}_h -  \Pi_h u^\a\|_{ E}  \le 2 \frac{\eta}{\gamma}.
\end{aligned}
\end{equation*}
Finally adding the best approximation error $\| \Pi_h u^\a -  u^\a\|_{H_*}$
gives the error estimate
\begin{displaymath}
\|u^{\rm ac}_h -  u^\a\|_{E} \leq
\| \Pi_h u^\a - u^\a\|_{E} + 2 \frac{\eta}{\gamma}.
\end{displaymath}

\subsection{Stability and Lipschitz condition}
\label{sec:prf-stab}
The Lipschitz and consistency estimates require bounds on the partial derivatives of $V$.  For $\bfg\in \R^{m\times 6}$, define the first
and second partial derivatives, for  $i,j = 1,\dots,6$, by
\begin{equation*}
\pp_jV(\bfg) := \frac{\pp V(\bfg)}{\pp g_j} \in \R^m, \quad \text{and} \quad \pp_{i,j}V(\bfg) : = \frac{\pp^2 V(\bfg)}{\pp g_i\pp g_j} \in \R^{m\times m},
\end{equation*}
and similarly for the third derivatives $\pp_{i,j,k}V(\bfg)\in \R^{m\times
m\times m}$. We assumed in \S~\ref{sec:atomistic-model} that second and higher
derivatives are bounded, hence we can define the constants
\begin{align}
M_2:& = \sum_{i,j=1}^6 \sup_{\bfg\in \R^{m\times 6}} \sup _{\substack{h_1,h_2 \in \R^2,\\|h_1| = |h_2| = 1}} \pp_{i,j} V(\bfg)[h_1,h_2] < \infty, \quad \text{and} \label{def:M_2}\\
M_3:& = \sum_{i,j,k=1}^6 \sup_{\bfg\in \R^{m\times 6}} \sup _{\substack{h_1,h_2,h_3 \in \R^2,\\ |h_1| = |h_2|=|h_3 |= 1}} \pp_{i,j,k} V(\bfg)[h_1,h_2,h_3] < \infty. \label{def:M_3}
\end{align}

With the above bounds it is easy to show that

\begin{equation} \label{eq:lip_V}
	\begin{split}
		\sum_{i=1}^6|\pp_i V(\bfg) - \pp_i V(\bfh)| &\le M_2 \max_{j = 1,\dotsc,6} |g_j-h_j|, \quad \text{and} \\
		\sum_{i,j=1}^6|\pp_i\pp_j V(\bfg) - \pp_i\pp_j V(\bfh)| &\le M_3 \max_{k = 1,\dotsc,6} |g_k-h_k|,
		\quad \text{for }\bfg,\bfh \in \R^{m\times 6}.
	\end{split}
\end{equation}

We can now obtain the following Lipschitz continuity and stability results.

\begin{lemma}
	\label{th:Lip}
	There exists $M>0$ such that
	\begin{equation}\label{eq:lip}
	\|\delta \mathcal{G}_h(u_h)-\delta \mathcal{G}_h(v_h)\|_\mathcal{L} \le M \| u_h- v_h\|_{E} \quad \text{for all } u_h,v_h \in \mathcal{U}^*_h,
	\end{equation}
	where $\|\cdot\|_{\mathcal{L}}$ denotes the operator norm associated
	with $\|\cdot\|_E$.
\end{lemma}
%
%


\section{Proofs: Consistency}

\subsection{Interpolants}
In this section we introduce two interpolants that are necessary tools
for our analysis.

\subsubsection{Test function $v$}\label{sec: test_fun_v}
The consistency error $\del \E^{\rm tot}_h(\Pi_h u^\a)$ will be bounded by estimating
\[
	\langle \del \E^{\rm tot}_h(\Pi_h u^\a), v_h \rangle
			- \langle \del \E^{\rm tot}(u^\a), v \rangle
	\leq \eta_h \| v_h \|_E \qquad \forall v_h \in \U^*_h,
\]
with $v \in \dot{\mathcal{U}}^{1,2}$ chosen arbitrarily.
The purpose of this section is to construct such $v = \Pi_h^* v_h$,
where $\Pi_h^* : \U_h^* \to \dot{\mathcal{U}}^{1,2}$.

Given some $v_h \in \U^* _h$ the first step is to extend $v_h$ to
 $\R^2$. Let $v^E_h$ be the solution to the  exterior Dirichlet
 problem
\begin{equation}\label{eq:lapv_h}
\begin{aligned}
-\Delta v^E_h &= 0, \quad \text{in } {\Omega}_h^\complement,\\
v^E_h & = v_h , \quad \text{on } \pp {\Omega}_h, \\
v^E_h & = v_h, \quad \text{in }\Omega_h,\\
|v^E_h(x)| &= \mathcal{O}\left({\frac{1}{|x|}}\right) \quad \text{as }|x|\rightarrow \infty,
\end{aligned}
\end{equation}
where we note that the last condition can be imposed because
$v_h \in \U_h^*$.

Next, we adapt the quasi-interpolation operator introduced in
\cite{carstensen} to ``project'' ${v}^E_h$ to $\Uss$. Let $\phi_\ell$ be the piecewise linear hat-functions on
the atomistic triangulation $\T$, i.e., the canonical triangulation associated with $\L$. Define
\begin{equation*}
\phi^{\rm PU}_\ell := \frac{\phi_\ell}{\sum_{k\in \Cs } \phi_k }  \qquad \forall \ell \in \Cs,
\end{equation*}
where $\Cs$ is the continuum lattice sites as defined in Section \ref{sec:g23_coupling}.
It is clear that $\{ \phi^{\rm PU}_\ell\}_{\ell \in \Cs }$ is a
partition of unity of $\R ^2\setminus (\Omega^\a \cup \Omega^\i )$.

In order to estimate the interpolation error and modelling error in \eqref{eq:consistency_decomposition}, we need $v-v^E_h$ to vanish in
$\Omega^\a \cup \Omega^\i $ and on $\Gamma_h$. This is made possible due to
assumption {\bf (A3)}.

Now we refer to \cite{carstensen} for the contruction of a linear interpolant of $v^E_h\in \U_h $ . We shall define the interpolant as follows:
\begin{equation}\label{def: varphi_intp}
\Pi_h^* v_h(x) :=  v(x) := v_1(x)+v_2(x), \quad \forall x \in \R ^2,
\end{equation}
where
\begin{equation*}
\begin{aligned}
v_1(\ell) &:= \cases{
	v_h(\ell),  & \ell \in \As \cup \Is \cup \Is^+ \cup (\Gamma_h \cap \Cs ), \\
	\frac{\int_{\R ^2}\phi_\ell v^E _h }{\int_{\R ^2 }\phi_\ell}, & \ell \in \Cs \setminus (\Is ^+\cup \Gamma_h),} \\
v_1(x) &:= \sum_{\ell \in \Lambda} v_1(\ell) \phi_\ell(x), \quad \forall x \in \R ^2,\\
v_2(\ell) &:= \cases{
	\frac{\int_{\R ^2 }(v^E _h-v_1) \phi_\ell^{\rm PU}}{\int_{\R ^2} \phi_\ell}, & \ell \in \Cs \setminus(\Is ^+\cup \Gamma_h),\\
	0, & \ell \in \As \cup \Is \cup \Is ^+ \cup (\Gamma_h \cap \Cs ),}\\
v_2(x) &:= \sum_{\ell \in \L} v_2(\ell) \phi_\ell(x), \quad \forall x \in \R ^2.
\end{aligned}
\end{equation*}

Note that with the assumption {\bf(A3)}, we have
\begin{equation*}
v(x) - v_h(x) = 0, \quad\forall x \in \Gamma_h.
\end{equation*}

We can use \cite[Theorem 3.1]{carstensen} to conclude that
\begin{equation*}
\|\D v\|_{L^2(\R ^2)} \lesssim \|\D v^E_h\|_{L^2(\R ^2)}, \quad \forall v_h \in \mathcal{U}_h .
\end{equation*}

Furthermore, since $v_h^E$ is the extension of $v_h$ via the exterior Laplace problem \eqref{eq:lapv_h}, we can link its energy norm to boundary norm of $v_h$. By the regularity of $g^{-1}$ in Lemma \ref{lem:g-1_rescale} we have
\begin{equation}\label{eq:eqv_v_v_h^E}
\|\D v^E_h\|^2_{L^2(\Omega_h^\complement)}  =  \<g^{-1} v_h, v_h\> \lesssim \|v_h\|^2_{H^{1/2}_{\Gamma_h}}.
\end{equation}
Therefore we have
\begin{equation} \label{eq: v-v_h^E -v_h}
\|\D v\|_{L^2(\R ^2)} \lesssim \|\D v_h\|_{L^2(\Omega_h)} + \|\D v^E_h\|_{L^2(\Omega_h^\complement )}
	\lesssim \|\D v_h\|_{L^2(\Omega_h)} + \|v_h\|_{H^{1/2}_{\Gamma_h}} \lesssim \|v_h\|_E.
\end{equation}

\subsubsection{Linearized elasticity approximation $w$}\label{sec:def_w}

Recall that $u^\a \in \dot{\mathcal{U}}^{1,2}$ is the exact atomistic solution
and Lemma \ref{lem:smooth_int} shows that there exists a $C^{2,1}$-regular
interpolant $\tilu^\a $ of $u ^\a $.

In order to make use of existing BEM approximation error estimates \eqref{eq:
g-g_h}, we need the conormal derivative in $H^{1}(\Gamma_h)$ of a solution
to Laplace's equation ($\tilu^\a$ only solves Laplace's equation approximately).
To that end, we introduce an intermediate problem on a domain with smooth boundary inside $\Omega_h$. Let $\mathcal{B}_R\subset \Omega_h$ be a ball with radius $R = \frac{2}{3} N$. To ensure the appropriate Dirichlet boundary condition, we use \eqref{eq:proj_to_*} to define the following function: let $u^\a_0$ be a constant such that
\begin{equation*}
u_R^\a : = \tilu ^\a -u^\a _0 \quad \text{ and } u_R^\a|_{\pp \mathcal{B}_R} \in H^{1/2}_*(\pp \mathcal{B}_R).
\end{equation*}
Let $w$ be the solution to the exterior Dirichlet problem
\begin{equation}\label{eq:lapw}
\begin{aligned}
-\Delta w &= 0, \quad \text{in } {\mathcal{B}}_R^\complement,\\
w & = u_R^\a , \quad \text{on } \pp \mathcal{B}_R, \\
|w(x)| &= \mathcal{O}\left({\frac{1}{|x|}}\right) \quad \text{as }|x|\rightarrow \infty.
\end{aligned}
\end{equation}

\begin{lemma} \label{lem:w-ua} The Dirichlet problem \eqref{eq:lapw} has a unique solution and
	\begin{equation}\label{eq:w-ua}
	\|\D w - \D \tilu ^\a \|_{L^2(\mathcal{B}_{R}^\complement)} \lesssim R^{-3}.
	\end{equation}

	\end{lemma}
\begin{proof}
	From Section \ref{sect: steklov-poincare} we know that this exterior
	Dirichlet problem has a unique solution. To estimate \eqref{eq:w-ua}, we let
	$\phi := \tilu ^\a - w$, extended by zero to $\mathcal{B}_{R}$ then
	\begin{align*}
		\|\D \phi \|_{L^2(\mathcal{B}_{R}^\complement)}^2
		&= \int_{\mathcal{B}_{R}^\complement}
					(\D \tilu^\a - \D w) \cdot \nabla \phi \\
		&= \int_{\R^2} \D \tilu^\a \cdot \nabla \phi =: B
	\end{align*}

	Next, we use the fact that $B$ is a linearised continuum approximation
	to the atomistic equilibrium equations. Recalling that $u^\a $ is an
	atomistic solution, i.e.,
	\begin{equation*}
	\< \del \E ^\a(u^\a), \phi\> = 0 \qquad \forall \phi \in \Uss,
	\end{equation*}
	and that $\pp^2_\mF W(0) = \mu I$, we can split $B$ into
	\begin{equation*}
	\begin{aligned}
	B 
	= &\int_{\R ^2} \D \tilu ^\a \cdot \D \phi   - \mu^{-1} \<\del \E ^\a(u^\a), \phi\>\\
	= &\left( \int_{\R ^2} \D \tilu ^\a \cdot \D \phi - \mu^{-1} \int_{\R ^2} \pp_\mF W(\D \tilu ^\a) \cdot \D \phi \right) \\
	& +  \mu^{-1}  \left(  \int_{\R ^2} \pp_\mF W(\D \tilu ^\a) \cdot \D \phi -  \<\del \E ^\a(u^\a), \phi\>\right) \\
	& =: B_1+B_2.
	\end{aligned}
	\end{equation*}
	For $B_1$, we apply Taylor's expansion and use
	$\pp^2_\mF W(0) = \mu I$ to obtain
	\begin{align*}
	|B_1|
	\leq& \bigg| \int_{\R ^2} \Big(\D \tilu^\a - \mu^{-1} \partial_F W(0)
				- \mu^{-1} \partial_F^2 W(0) \D \tilu^\a\Big) \cdot \D \phi \bigg|
				 	+ C \int_{\R^2} |D\tilu^\a|^2\,|\D\phi| \\
	=&  C \int_{\R^2} |\D\tilu^\a|^2\,|\D\phi|
	\leq C \| \D\tilu^\a \|_{L^4(\mathcal{B}_{R}^\complement)}^2
			\| \D\phi \|_{L^2}
	\leq C R^{-3} \| \D\phi \|_{L^2},   
	\end{align*}
	where the constant $C$ is independent of $\tilu ^\a $ and $\phi$.

	$B_2$ is the Cauchy--Born modelling error
	 which is well understood, e.g., in \cite{PRE-ac.2dcorners} it is proven
	 that
	\begin{equation*}
		B_2 \le \int_{\R^2} (C_1 |\D ^3\tilu ^\a|+ C_2|\D ^2\tilu ^\a |^2
	) |\D \phi|,
	\end{equation*}
	hence we obtain
	\begin{align*}
		B_2 &\leq \Big(\| \D^3 \tilu^\a \|_{L^2(\mathcal{B}_{R}^\complement)}
				+ \| \D^2 \tilu^\a \|_{L^4(\mathcal{B}_{R}^\complement)}^2 \Big)
				\| \D\phi \|_{L^2} \\
			&\leq \big( R^{-3} + R^{-5} \big)	\| \D\phi \|_{L^2}.
	\end{align*}
	Combining the estimates for $B_1$ and $B_2$ yields the stated result.
\end{proof}

The second estimate we require for $w$ is for the decay of $\D^2 w$.

\begin{lemma} \label{th:decay_D2w}
	Let $w$ be given by \eqref{eq:lapw}, and $R \leq \frac23 N$, where $N$ is the inner radius of $\Omega_h$, then
	\begin{equation} \label{eq:decay_D2w}
		|\D^2 w(x)| \lesssim |x|^{-3} \qquad \text{for } |x| \geq N
	\end{equation}
	and in particular,
	\begin{equation} \label{eq:norm_D2w}
		\| \D^2 w \|_{L^2(\Omega_h^\complement)}
			\lesssim N^{-5/2}
	\end{equation}
\end{lemma}
\begin{proof}
Since the auxiliary problem \eqref{eq:lapw} involves a circular
boundary $\pp \Omega_R^\complement$, we can exploit separation of variables and Fourier series to
estimate $\D^2 w$. We write $\tilu ^\a$ and $w$ in polar coordinates as
\begin{align}
\tilu ^\a (r, \theta) &=  \sum_{k\in\Z } \hat{a}_k(r) e^{ik\cdot \theta} ,\notag \\
w(r,\theta) &= \sum_{k\in\Z } \hat{W}_k(r) e^{ik\cdot \theta}. \label{eq:fouriersw}
\end{align}
The boundary condition $w = \tilu^\a$ on $\partial\mathcal{B}_R$ becomes
\begin{equation*}
w(R, \theta) = \tilu^\a(R,\theta), \quad \text{i.e.,} \qquad
 \hat{W}_k(R) = \hat{a}_k(R).
\end{equation*}

The Laplace operator in polar coordinates in 2D is given by
\begin{equation*}
-\Delta_{x,y} = -r^{-1}\pp _r(r^{-1} \pp _r ) -r^{-2}\pp ^2_\theta.
\end{equation*}
Substituting \eqref{eq:fouriersw} we obtain
\begin{equation*}
\sum_{k\in \Z} \pp^2_r \hat{W}_k + r^{-1} \pp_r\hat{W}_k - k^2 r^{-2} \hat{W}_k = 0.
\end{equation*}
Solving the resulting ODE for each $\hat{W}_k$ and taking into account the decay
 and boundary condition from \eqref{eq:lapw}, we deduce that
\begin{equation*}
w(r,\theta) = \sum_{k\in\Z } \hat{a}_k(R) \left(\frac{r}{R}\right)^{-|k|} e^{ik\cdot \theta}.
\end{equation*}

Using the fact that, for $p \in \mathbb{N}$ and $q \geq 1+\epsilon$,
\begin{equation*}
	\sum_{k\in\Z } |k|^p q^{-2|k|} \leq C_{p,\epsilon} q^2.
\end{equation*}

%
%
We can now estimate
\begin{align*}
\big|\D_r ^2 w(r, \theta)\big|
&= \bigg|r^{-2}\sum_{k\in\Z } |k|(|k|+1) \hat{a}_k(R) \left(\frac{r}{R}\right)^{-|k|} e^{ik\cdot \theta}\bigg| \\
& \le r^{-2} \left(\sum_{k\in\Z }  |\hat{a}_k(R)|^2\right)^{1/2} \left(\sum_{k\in\Z } |k|^4 \left (\frac{r}{R}\right)^{-2|k|} \right)^{1/2} \\
&\lesssim r^{-2} (r/R)^2 \bigg( \frac1R \int_{\partial \mathcal{B}_R} |\tilu^\a|^2 \bigg)^{1/2},
\end{align*}
where in the last line we also used Plancherel's Theorem. Using the fact that
$|\tilu^\a(x)| \lesssim |x|^{-1}$ we finally obtain
\begin{equation*}
	\big|\D_r^2 w(r, \theta)\big| \lesssim (r/R)^{-3}.
\end{equation*}
Analogous arguments for $\D_\theta^2 w$ and $\D_r\D_\theta w$ yield
\begin{equation*}
	\big|\D^2 w(r, \theta)\big| \lesssim (r/R)^{-3}.
\end{equation*}

The first result \eqref{eq:decay_D2w} follows from the assumption that
$R \leq \frac23 N$. The second result \eqref{eq:norm_D2w} is an immediate
consequence of \eqref{eq:decay_D2w}.
\end{proof}

\subsection{Consistency decomposition}
Given a solution $u^\a$ to \eqref{eq:y_a} and a discrete test function
$v_h$, let $\Pi_h u^\a$ be as defined in \S~\ref{sect:best_approx}, let
$v = \Pi_h^* v_h$ be defined by \eqref{def: varphi_intp}, and let
$w$ be given by \eqref{eq:lapw}. Moreover, let
\begin{equation*}
	\tilu^\a _h : = \tilu^\a - c_h \quad \text{such that} \quad
 	\tilu^\a _h|_{\Gamma_h} \in H^{1/2}_*(\Gamma_h),
\end{equation*}
then we decompose the consistency error into
 \begin{align}
	 \notag
	 \<\del \E_h ^{\rm tot }(\Pi_h u^\a), v_h\>
	 &=
 \<\del \E_h ^{\rm tot }(\Pi_h u^\a), v_h\> - \<\del \E ^\a (u^\a),v\>   \\
 \notag
 &= \<\del\E^{\rm ac}_h(\Pi_h u^\a ), v_h\>+ {\mu} \<  g_h^{-1} \Pi_h u^\a,v_h\>_{\Gamma_h} -  \<\del \E ^\a (u^\a),v\> \\
 \label{eq:consistency_decomposition}
 &= \underbrace{  \<\del\E^{\rm ac}_h(\Pi_h u ), v_h\> - \<\del \E_h^{\rm ac}(\tilu^\a ), v\>}_{\text{interpolation error}}\\
 \notag
 &\qquad + \underbrace{
 	\<\del \E ^{\rm ac} (\tilu^\a) - \del \E ^\a (\tilu^\a), v\> -  \int_{\Omega_h^\complement} (\pp W(\D \tilu^\a ) - \mu \D \tilu^\a )\cdot \D v}_{\text{modelling error}}\\
	\notag
 &\qquad + \underbrace{\mu\<g^{-1}_h \Pi_h u^\a, v_h\>_{\Gamma_h} -\mu \int_{\Omega_h^\complement} \D \tilu^\a \cdot \D v}_{\text{BEM error}}.
\end{align}

 \subsection{The interpolation error}
The first part of the consistency error, the interpolation error,
has already been estimated in \cite{2016-qcp2}.

\begin{lemma}\label{lem:int_err}
 	The interpolation error can be estimated by
 	\begin{equation}\label{eq:int_err}
 	\begin{aligned}
&\<\del\E^{\rm ac}_h(\Pi_h u^\a ), v_h\> - \<\del \E_h^{\rm ac}(\tilu^\a ), v\> \\
&\quad \le c(M_2\|h\D ^2 \tilu ^\a\|_{L^2(\Omega_h\setminus \Omega^\a )}+ M_2\|\D ^3\tilu^\a\|_{L^2(\R ^2\setminus \Omega^\a)} + M_3\|\D ^2 \tilu ^\a \|^2_{L^4(\R ^2\setminus \Omega^\a )})\|\D v_h\|_{L^2}.
 	\end{aligned}
 	 	\end{equation}
 	\end{lemma}
 \begin{proof}
 	We  split the interpolation error into
 	\begin{align*}
 	\<\del\E^{\rm ac}_h(\Pi_h u^\a ), v_h\> - \<\del \E_h^{\rm ac}(\tilu^\a ), v\>
	= \<\del\E^{\rm ac}_h(\Pi_h u^\a) - \del\E^{\rm ac}_h(\tilu^\a ), v_h\>
	- \<\del \E_h^{\rm ac}(\tilu^\a ), v-v_h\>.
 	\end{align*}
 	The first term can be bounded by a standard interpolation error estimate
	and the uniform boundedness of $\del^2\E^{\rm ac}_h$,
 	\begin{align*}
 	\<\del\E^{\rm ac}_h(\Pi_h u^\a) - \del\E^{\rm ac}_h(\tilu^\a ), v_h\> &\le \<\del^2\E^{\rm ac}_h(\theta  )(\D \Pi_h u^\a - \D \tilu^\a), v_h\> \\
 	&\le c M_2 \|h\D ^2 \tilu ^\a\|_{L^2(\Omega^\c_h )}\|\D v_h\|_{L^2(\Omega^\c _h)}
 	\end{align*}

 	The bound for the second term follows from the exactly same argument as in the proof of \cite[Theorem 3.2]{2016-qcp2}. Since the interpolant $v$ defined in \eqref{def: varphi_intp} has property
 	\begin{equation*}
 	v(x) - v_h(x) = 0 , \quad  \forall x \in \Gamma_h \cup \Omega^\a \cup \Omega^\i
 	\end{equation*}
we can integrate by part in $\Omega_h^\c$ without obtaining boundary contributions. Let $Q : = -{\rm div}\, [\pp _{\mF}W(\D \tilde{u}^{\rm a})]$, then
\begin{equation*}
\<\del \E_h^{\rm ac}(\tilu^\a ), v-v_h\>
= \int_{\Omega_h^\c} Q \cdot (v_h-v) \dx = \int_{\Omega_h^\c} Q \cdot \left((v_h-v_1)- v_2\right) \dx.
\end{equation*}
Since $v_2$ is a piecewise-linear quasi-interpolant of $v_h - v_1$ as defined in \cite{carstensen}, a direct consequence of Theorem 3.1 in \cite{carstensen} is that there exists $C>0$ such that,
\begin{equation*}
	\<\del\E^{\rm ac}_h(\Pi_h u^\a) - \del\E^{\rm ac}_h(\tilu^\a ), v_h\> \le C\|\D (v_h-v_1) \|_{L^2(\Omega^\c _h)}\left( \sum_{\ell \in \Cs \cap \Omega^\c_h} d_\ell ^2 \int_{w_\ell} \phi^{\rm PU}_\ell |Q - \<Q\>_\ell|^2 \dx\right)^{1/2},
\end{equation*}
where $w_\ell : = \supp (\phi_\ell)$,
$\<Q \>_\ell := 1/|w_\ell| \int_{w_\ell} Q(x)\dx$ and
$d_\ell :={\rm diam}(w_\ell) = 1$. With the sharp Poincar\'{e} constant derived
in \cite{Acosta2003}, we obtain
\begin{equation*}
\int_{w_\ell } \phi^{\rm PU}_\ell |Q - \<Q\>_\ell|^2 \dx \le \int_{w_\ell} |Q-\<Q\>_\ell|^2 \dx\le \tfrac{1}{4} d^2_\ell \|\D Q\|^2_{L^2(w_\ell)}.
\end{equation*}
On the other hand, $v_1$ is a standard quasi-interpolant of $v_h$ in $ \bigcup{\mathcal{T}_h^c}$, which implies that there exists $C' >0$ such that
\begin{equation*}
\|\D (v_h - v_1)\|_{L^2(\Omega^\c _h)}  \le C' \|\D v_h\|_{L^2(\Omega^\c _h)}.
\end{equation*}

Due to the fact that $d_\ell = 1$ and that each point in
$\R^2 \setminus \Omega^\a$ is covered by at most three $w_\ell$, we have
\begin{align*}
	&\<\del\E^{\rm ac}_h(\Pi_h u^\a) - \del\E^{\rm ac}_h(\tilu^\a ), v_h\>\\ \le &C \max_\ell d_\ell^2 \|\D Q\|_{L^2(\Omega^\c _h)}\|\D v_h\|_{L^2(\Omega^\c _h)}  \\
\le &C\left(M_2\|\D^3 \tilde{u}^\a\|_{L^2(\Omega^\c _h)}+M_3\|\D^2 \tilde{u}^\a \|^2_{L^4(\Omega^\c _h)}\right)\|\D v_h\|_{L^2(\Omega^\c _h)},
\end{align*}
where we used the following estimate, for some $c>0$,
\begin{equation*}
\begin{aligned}
\|\D Q  \|_{L^2(\Omega_h)} &= \|\D {\rm div} [\pp _\mF W (\D \tilde{u}^\a)]\|_{L^2(\Omega^\c _h)}  \\
&= \|\D \left(\pp_\mF ^2 W (\D \tilde{u}^\a) \D ^2 \tilde{u}^\a \right)\|_{L^2(\Omega^\c _h)} \\
&= \left\| \pp_\mF ^2 W (\D \tilde{u}^\a) \D ^3 \tilde{u}^\a +  \pp_\mF ^3 W (\D \tilde{u}^\a) \left(\D ^2 \tilde{u}^\a \right)^2\right\|_{L^2(\Omega^\c _h)}  \\
&\le c \left(M_2\|\D^3 \tilde{u}^\a\|_{L^2(\Omega^\c _h)}+M_3\|\D^2 \tilde{u}^\a \|^2_{L^4(\Omega^\c _h)}\right),
\end{aligned}
\end{equation*}
employing the global bounds \eqref{def:M_2} and \eqref{def:M_3}.
\end{proof}

 \subsection{The modelling error}
 In this section we rely on the following theorem from \cite{PRE-ac.2dcorners} of the pure modelling error estimate of G23 coupling method.

 \begin{theorem}[G23 modeling error]
 	For any $v\in \Uss$ we have the G23 consistency error
 	\begin{equation}
 	\begin{aligned}\label{thm:g23_consist_pure}
 	&\<\del \E ^{\rm ac} (u^\a ),v\> -\<\del \E ^\a (u^\a ), v\> \\
 	&\quad \le c\left(M_2\|\D^2\tilu ^\a \|_{L^2(\Omega^{\rm i})} + M_2\|\D^3 \tilu ^\a \|_{L^2(\Omega^\c)}+M_3\|\D ^2\tilu ^\a\|^2_{L^4(\Omega^\c)}\right)\|\D v\|_{L^2(\R ^2)}
 	\end{aligned}
 	\end{equation}
 \end{theorem}

 Furthermore, the second term of the modelling error can be estimated as follows.

 \begin{lemma}
	For any $v\in \Uss$ we have
 	\begin{equation*}
 	\int_{\Omega_h^\complement} (\pp W(\D \tilu ) - \mu \D \tilu )\cdot \D v \lesssim \|\D \tilu ^\a \|^2_{L^4(\Omega_h^\complement)}\|\D v\|_{L^2(\Omega_h^\complement)}.
 	\end{equation*}
 \end{lemma}
 \begin{proof}
 	This is a direct result from applying Taylor expansion,
 	\begin{align*}
 	\int_{\Omega_h^\complement} (\pp W(\D \tilu ) - \mu \D \tilu )\cdot \D v  &= \int_{\Omega_h^\complement}\left[\pp W(0) + \pp ^2 W(0) \D \tilu ^\a + \frac{1}{2}\pp ^3 W(\theta )(\D \tilu^\a)^2 \right]\D v\\
 	& \quad - \int_{\Omega_h^\complement}  \mu \D \tilu ^\a \D v\\
 	& \le \frac{M_3}{2} \int_{\Omega_h^\complement}  (\D \tilu ^\a)^2 \D v\\
 	&\lesssim \|\D \tilu ^\a \|^2_{L^4(\Omega_h^\complement)}\|\D v\|_{L^2(\Omega_h^\complement)},
 	\end{align*}
 	where we use the fact that $\pp W (0) = 0$ and that $\pp^2 W (0)  = \mu I$.
 \end{proof}
\smallskip

Therefore the modelling error can be estimated by
\begin{equation}\label{eq:modelling_err}
\begin{aligned}
	&\<\del \E ^{\rm ac} (\tilu^\a) - \del \E ^\a (\tilu^\a), v\> -  \int_{\Omega_h^\complement} (\pp W(\D \tilu^\a ) - \mu \D \tilu^\a )\cdot \D v\\
	\lesssim & \left(\|\D^2\tilu ^\a \|_{L^2(\Omega^{\rm i})} + \|\D^3 \tilu ^\a \|_{L^2(\R ^2 \setminus \Omega^\a)}+\|\D ^2\tilu ^\a\|^2_{L^4(\R ^2 \setminus \Omega^\a)}\right)\|\D v\|_{\R ^2} \\
	\lesssim & \left(\|\D^2\tilu ^\a \|_{L^2(\Omega^{\rm i})} + \|\D^3 \tilu ^\a \|_{L^2(\R ^2 \setminus \Omega^\a)}+\|\D ^2\tilu ^\a\|^2_{L^4(\R ^2 \setminus \Omega^\a)}\right) \| v_h\|_{E},
\end{aligned}
\end{equation}
where we used the face that $\|\D v\|_{L^2(\R ^2)} \lesssim \|v_h\|_E$ from \eqref{eq: v-v_h^E -v_h}.

\subsection{The BEM error}
To complete the analysis of our numerical scheme it remains to estimate
the BEM error contribution to the consistency error \eqref{eq:consistency_decomposition}. Recall that we need to estimate
\[
	\< g^{-1}_h \Pi_h u^\a, v_h\>  -\int_{\Omega_h^\complement} \D \tilu^\a \cdot \D v,
\]
where $v$ is the interpolant defined in Section \ref{sec: test_fun_v}. Recall that $v_h^E$ solves the exterior Laplace problem \eqref{eq:lapv_h}. Then we have
\begin{equation*}
\int_{\Omega_h^\complement} \D \tilu^\a \cdot \D v_h^E = \< g^{-1} \tilu^\a _h, v_h\>_{\Gamma_h}.
\end{equation*}
Then the BEM error can be decomposed into
\begin{equation*}
	\begin{aligned}
	\< g^{-1}_h \Pi_h u^\a, v_h\>  -\int_{\Omega_h^\complement} \D \tilu^\a \cdot \D v & =\< g^{-1}_h \Pi_h u^\a, v_h\> - \< g^{-1} \tilu^\a_h, v_h\> + \int_{\Omega_h^\complement} \D \tilu^\a \cdot(\D v_h^E -  \D v )\\
	&=  \<g_h^{-1}(\Pi_h u^\a  - \tilu^\a_h ),v_h\> +  \<(g_h^{-1} - g^{-1})w, v_h\> \\
	&\quad + \<g^{-1}(w- \tilu ^\a_h ), v_h\> + \<g_h^{-1}(\tilu^\a_h  - w) ,v_h\>\\
	&\quad  + \int_{\Omega_h^\complement} \D \tilu^\a \cdot(\D v_h^E -  \D v )\\
	& =: A_1+A_2+A_3+A_4+A_5,
	\end{aligned}
	\end{equation*}
where we use the fact that $v_h = v_h^E $ on $\Gamma_h$.

We will employ stability of $g_h^{-1}$ and $g^{-1}$, as stated in Theorems \ref{thm:bdry_bdd} and \ref{thm: ellip}. In addition, the estimate of $A_1$ relies on best approximation error bounds; $A_2$ is the standard BEM approximation error;   $A_3$ and $A_4$ require the results on the auxiliary function $w$ that we established in \S~\ref{sec:def_w}; while estimating $A_5$ is analogous of the proof of Lemma \ref{lem:int_err}.

\subsubsection{Estimate of $A_1$}

In this section we first discuss the best approximation error $\|\D \Pi_h u^\a - \D \tilu^\a\|_{L^2(\Gamma_h)}$. We will exploit the theorems below, which are well established in the literature.

\begin{theorem}[Interpolation]
	Recall that the rescaled norms $H^{1/2}_{\Gamma_h}$ and $H^1_{\Gamma_h}$ are defined in \eqref{def:rescale_norm} and \eqref{def:rescale_norm_h1}, respectively. Let $u\in H^1_{\Gamma_h}$ then we have
	\begin{equation*}
	\|u\|_{H^{1/2}_{\Gamma_h}} \le \|u\|^{1/2}_{H^{1}_{\Gamma_h}}\|u\|^{1/2}_{H^0_{\Gamma_h}}.
	\end{equation*}
\end{theorem}
\vspace{-5mm}
\begin{proof} Let $u_1(x) := u(Rx)$ and $\Gamma_1$ be the image of mapping $\Gamma_h\ni x\mapsto \tfrac{x}{R}$.  Then, by definitions \eqref{def:rescale_norm} and \eqref{def:rescale_norm_h1}, we have
	\begin{align*}
	\|u\|_{H^{1/2}_{\Gamma_h}} & = \|u_1\|_{H^{1/2}(\Gamma_1)},\\
	\|u\|_{H^{1}_{\Gamma_h}} & = \|u_1\|_{H^{1}(\Gamma_1)}, \quad \text{and}\\
	\|u\|_{H^{0}_{\Gamma_h}} & = \|u_1\|_{H^{0}(\Gamma_1)}
	\end{align*}

	The standard interpolation theorem (see, for example, Theorem 2.18 in \cite{Olaf:BEM}) states that
		\begin{equation*}
	\|u_1\|_{H^{1/2}(\Gamma_1)} \le \|u_1\|^{1/2}_{H^{1}(\Gamma_1)}\|u_1\|^{1/2}_{H^0(\Gamma_1)}.
	\end{equation*}
	Hence the result follows.
\end{proof}

\begin{theorem}
	Recall that $\Pi_h$ was defined in Section \ref{sect:best_approx} as the piecewise linear nodal interpolation operator , then we have, for $v\in H^2(\Gamma_h)$,
	\begin{align*}
	\|v-\Pi_hv\|_{L^2(\Gamma_h)}&\le c\|h^2\D^2 v\|_{L^2(\Gamma_h)},\quad \text{and}\\
	\|v-\Pi_hv\|_{H^1(\Gamma_h)} & \le c\|h\D^2 v\|_{L^2(\Gamma_h)}
	\end{align*}
\end{theorem}
\begin{proof}
	This is a direct result of Bramble-Hilbert Lemma. It is worth noting
	that in fact we only need the tangential part of $\nabla^2 v$ for this
	estimate.
\end{proof}

Thus, using also $\big[\tfrac{1}{2}{\rm diam}(\Gamma_h)\big] \approx N$,
we can conclude that
\begin{equation}\label{eq:best_approx_err_H1/2}
\begin{aligned}
\|\Pi_h u^\a  -\tilu^\a_h  \|_{H^{1/2}_{\Gamma_h}}
&\le\|\Pi_h u^\a  -\tilu^\a_h \|^{1/2}_{H^1_{\Gamma_h}}\|\Pi_h u^\a  -\tilu^\a_h \|^{1/2}_{H^0_{\Gamma_h}} \\
&\lesssim \left( N \|h\D ^2 \tilu^\a_h\|^2_{L^2(\Gamma_h)} + N^{-1} \|h^2\D ^2 \tilu^\a_h\|^2_{L^2(\Gamma_h)}\right)^{1/4} \\
&\quad \cdot \left(N^{-1} \|h^2\D ^2 \tilu^\a_h\|^2_{L^2(\Gamma_h)}\right)^{1/4} \\
&\lesssim \|h\D ^2 \tilu^\a_h\|_{L^2(\Gamma_h)}^{1/2} \|h^2\D ^2 \tilu^\a_h\|_{L^2(\Gamma_h)}^{1/2}
 	+ N^{-2} \|h^2\D ^2 \tilu^\a_h\|_{L^2(\Gamma_h)} \\
&\lesssim \|h^{3/2} \D ^2 \tilu^\a_h\|_{L^2(\Gamma_h)},
\end{aligned}
\end{equation}
where, in the last line, we used the fact that $N^{-2} h \lesssim 1$
and that $h$ is quasi-uniform on $\Gamma_h$.

By the stability estimate \eqref{eq: g_h_stability}, we have
\begin{equation*}
\|g_h^{-1}(\Pi_h u -\tilu^\a_h)\|_{H^{-1/2}(\Gamma_h)} \le c\|\Pi_h u -\tilu^\a_h\|_{H^{1/2}_{\Gamma_h}}.
\end{equation*}
Therefore we can estimate $A_1$ as follows
\begin{equation}\label{eq:A1}
A_1 \le C\|\Pi_h u^\a  -\tilu^\a_h  \|_{H^{1/2}_{\Gamma_h}} \|v_h\|_{H^{1/2}_{\Gamma_h}}  \lesssim \|h^{3/2}\D ^2\tilu^\a  \|_{L^2(\Gamma_h)} \|v_h\|_{H^{1/2}_{\Gamma_h}}.
\end{equation}

\subsubsection{Estimate of $A_2$}
Since $w$ is the solution to the Laplace equation \eqref{eq:lapw} with smooth boundary $\pp \mathcal{B}_R$, its conormal derivative $g^{-1}w$ on $\Gamma_h$ is in $H^{1}(\Gamma_h)$. Hence we can apply Theorem \ref{thm:g-g_h}
and then Lemma \ref{lem:g-1_rescale} to estimate
\begin{align}
		\notag
A_2 &\le \|(g^{-1}_h - g^{-1} ) w\|_{H^{-1/2}(\Gamma_h)} \|v_h\|_{H^{1/2}_{\Gamma_h}}\\
\notag
 &\le c h^{3/2}\|\D (g^{-1}w)\|_{L^2(\Gamma_h)}\|v_h\|_{H^{1/2}_{\Gamma_h}}\\
 \notag
 &\lesssim h^{3/2} \|\D^2 w\|_{L^2(\Gamma_h)}\|v_h\|_{H^{1/2}_{\Gamma_h}}\\
 	\label{eq:estimate_A2}
 &\lesssim h^{3/2} N^{-5/2}\|v_h\|_{H^{1/2}_{\Gamma_h}},
\end{align}
where the last line results from \eqref{eq:decay_D2w} and the Trace Theorem.

\subsubsection{Estimates of $A_3$ and $A_4$ }\label{sec:A3A4}
By the stability of $g^{-1}$, we have
\begin{equation*}
A_3 \le C_2 \|\tilu^\a_h  - w\|_{H^{1/2}_{\Gamma_h}}\|v_h\|_{H^{1/2}_{\Gamma_h}}.
\end{equation*}
Since $g^{-1}$ is positive-definite and bounded by Lemma \ref{lem:g-1_rescale}, we can link $\|\cdot \|_{H^{1/2}_{\Gamma_h}}$ to $\|\D \cdot \|_{L^2(\Omega_h^\complement)}$ through exterior Laplace problems.

Recall that $\tilu^\a _h : = \tilu^\a  - c_h \in H^{1/2}_*(\Gamma_h)$. Then by Theorem \ref{thm:+ve_def} the following exterior Laplace problem has a unique solution
\begin{equation}\label{eq:lap_y}
\begin{aligned}
-\Delta y &= 0, \quad \text{in } \Omega_h^\complement,\\
y & = \tilu^\a_h , \quad \text{on } \Gamma_h, \\
|y(x)| &= \mathcal{O}\left({\frac{1}{|x|}}\right) \quad \text{as }|x|\rightarrow \infty.
\end{aligned}
\end{equation}
Then arguing exactly as in the proof of Lemma \ref{lem:w-ua}, we have
\begin{equation*}
\|\D y - \D \tilu ^\a \|_{L^2(\Omega_h^\complement)} \lesssim N^{-3}.
\end{equation*}
In addition, by the positive-definiteness of $g^{-1}$ in Lemma \ref{lem:g-1_rescale} we have
\begin{align*}
C_1\|\tilu^\a_h  - w\|^2_{H^{1/2}_{\Gamma_h}} &\le  \< g^{-1} (\tilu^\a_h  - w), (\tilu^\a_h  - w)\>  = \int_{\Omega_h^\complement} |\D y - \D w|^2.
\end{align*}
Therefore we have
\begin{align}
\notag
A_3 &\lesssim\|\tilu^\a_h  - w\|_{H^{1/2}_{\Gamma_h}}\|v_h\|_{H^{1/2}_{\Gamma_h}}\\
\notag
&\lesssim \left(\|\D y - \D \tilu ^\a_h \|_{L^2(\Omega_h^\complement)} + \|\D w - \D \tilu ^\a_h \|_{L^2(\Omega_h^\complement)} \right) \|v_h\|_{H^{1/2}_{\Gamma_h}}\\
\notag
& \lesssim \left(N^{-3}+R^{-3} \right)\|v_h\|_{H^{1/2}_{\Gamma_h}}\\
\label{eq:estimate_A3}
&\lesssim N^{-3}\|v_h\|_{H^{1/2}_{\Gamma_h}}.
\end{align}

For $A_4$, using the stability of $g^{-1}_h$ in \eqref{eq: g_h_stability} and the same argument as for $A_4$, we have
\begin{equation}\label{eq:estimate_A4}
A_4 \lesssim N^{-3}\|v_h\|_{H^{1/2}_{\Gamma_h}}.
\end{equation}

\subsubsection{Estimate of $A_5$}\label{sec:A5}
Now we consider
\begin{equation*}
A_5 = \int_{\Omega_h^\complement} \D \tilu^\a \cdot(\D v_h^E -  \D v ).
\end{equation*}
Recall that $v$ is the quasi-interpolant of $v_h^E$ defined in \eqref{def: varphi_intp}. Under the assumption {\bf (A3)} we know that
\begin{equation*}
v_h^E (x) - v(x) = 0 \quad \forall x \in \Gamma_h.
\end{equation*}
So we can use analogous argument to the proof of Lemma \ref{lem:int_err} to get
\begin{align*}
A_5 & \lesssim \|\D {\rm div} [\D \tilde{u}^\a]\|_{L^2(\Omega_h^\complement)} \|\D v^E_h\|_{L^2(\Omega_h^\complement)} \\
& \lesssim \|\D^3 \tilde{u}^\a\|_{L^2(\Omega_h^\complement)} \|\D v^E_h\|_{L^2(\Omega_h^\complement)} .
\end{align*}
By \eqref{eq:eqv_v_v_h^E} we have
\begin{equation*}
 \|\D v^E_h\|_{L^2(\Omega_h^\complement)}   \lesssim \|v_h\|_{H^{1/2}_{\Gamma_h}}.
\end{equation*}
Therefore we have
\begin{equation}\label{eq:estimate_A5}
A_5 \lesssim \|\D^3 \tilde{u}^\a\|_{L^2(\Omega_h^\complement)} \|v_h\|_{H^{1/2}_{\Gamma_h}} \lesssim N^{-3}\|v_h\|_{H^{1/2}_{\Gamma_h}}.
\end{equation}

Summarising all five components of the BEM error estimates \eqref{eq:A1}, \eqref{eq:estimate_A2}, \eqref{eq:estimate_A3}, \eqref{eq:estimate_A4} and \eqref{eq:estimate_A5} we obtain
\begin{equation}\label{eq:BEM_err}
\text{BEM error} \lesssim \big(N^{-3} + \|
h^{3/2} \D ^2\tilu ^\a \|_{L^2(\Gamma_h)}\big) \|v_h\|_{H^{1/2}_{\Gamma_h}}.
\end{equation}

\subsection{Proof of Theorem \ref{thm:cons_main}}\label{sect:proof_thm_consit}
Finally, recalling the decomposition in \eqref{eq:consistency_decomposition}, we add the estimates for all three components \eqref{eq:int_err}, \eqref{eq:modelling_err} and \eqref{eq:BEM_err} together to get the following estimate.
	We have, for any $v_h\in \mathcal{U}^*_h$
	\begin{equation*}
	\begin{aligned}
	&\<\del \E_h ^{\rm tot }(\Pi_h u^\a), v_h\> \\
	 \lesssim &(\|\D ^2\tilu ^\a \|_{L^2(\Omega^\i) } +\|\D ^3\tilu ^\a \|_{L^2(\R ^2 \setminus \Omega ^\a )} +\|\D ^2\tilu ^\a \|^2_{L^4(\R ^2 \setminus \Omega ^\a )} )\|v_h\|_E\\
	 &\quad + \|h\D ^2\tilu ^\a \|_{L^2(\Omega^\c _h)} \|\D v_h\|_{L^2(\Omega^\c _h)}+ N^{-3} \|\D v_h\|_{H^{1/2}_{\Gamma_h}}\\
	 &\quad + \|
	 h^{3/2} \D ^2\tilu ^\a \|_{L^2(\Gamma_h)} \|v_h\|_{H^{1/2}_{\Gamma_h}}.
		\end{aligned}
	\end{equation*}
Therefore the result follows.

\subsection {Proof of Theorem \ref{thm:err_2}}\label{sect:proof_thm_err_2}
We shall use the Inverse Function Theorem \ref{theo:inverse}. To put into the context of Theorem \ref{theo:inverse}, let
\begin{equation*}
\mathcal{G}_h(v) : =  \del \E^{\rm tot}_h(v) -\del f(v) \quad \text{and}\quad \bar{u}_h := \Pi_h u^\a.
\end{equation*}
Then Theorem \ref{thm:cons_main} gives property \eqref{inv_consist} and Theorem \ref{thm:stab} gives property \eqref{inv_stab}. Then we can conclude that, for $K, N$ sufficiently large, there exists $u_h\in \mathcal{U}_h^*$ such that
\begin{equation*}
\mathcal{G}_h(u_h)  = 0, \quad \text{and}
\end{equation*}
\begin{align*}
\|u_h-\Pi_h u^\a\|_E &\lesssim \|\D ^2\tilu ^\a \|_{L^2(\Omega^\i) } +\|\D ^3\tilu ^\a \|_{L^2(\R ^2 \setminus \Omega ^\a )} +\|\D ^2\tilu ^\a \|^2_{L^4(\R ^2 \setminus \Omega ^\a )} \\
&\quad + \|h\D ^2\tilu ^\a \|_{L^2(\Omega^\c _h)}  + \|h^{3/2}\D ^2\tilu ^\a \|_{L^2(\Gamma_h)} + N^{-3} .
\end{align*}
Finally we add the best approximation error
\begin{align*}
\|\Pi_h u^\a - u_h\|^2_E &= \|\D \Pi_hu^\a - \D u_h\|^2_{L^2(\Omega_h)} +\|\Pi_hu^\a - u_h\|^2_{H^{1/2}_{\Gamma_h}}\\
& \lesssim \|h\D ^2\tilu ^\a \|^2_{L^2(\Omega^\c _h)} +  \|h^{3/2}\D ^2\tilu ^\a \|^2_{L^2(\Gamma_h)},
\end{align*}
where the last term comes from \eqref{eq:best_approx_err_H1/2}. Thus the result follows.

\appendix
\section{Proof of Theorem \ref{thm: ellip}}\label{app:ellip}
The proof follows exactly as Theorem 2 in \cite{Costabel} but with details specific for 2D, showing how the subspace $ H^{-1/2}_*(\Gamma_h)$ ensures that the far-field value $u_0 = 0$.

To construct the proofs for Theorem \ref{thm: ellip}, we need several intermediate results from literature.

\begin{lemma}\label{lem:decay_H*}
	Suppose $v\in H_*^{-1/2}(\Gamma_h)$ , $y_0\in \Omega_h$ and $u(x) = (Av)(x)$ for $x\in \R ^2\setminus\Gamma_h$, then we have
	\begin{align*}
	|u(x)| &\le c_1 \frac{1}{|x-y_0|} \quad \text{and}\\
	|\D u(x)| &\le c_2 \frac{1}{|x-y_0|^2} , \quad \text{for } |x-y_0|> \max \{1,2{\rm diam}(\Omega_h)\}.
	\end{align*}
\end{lemma}
\begin{proof}
	See Lemma 6.21 in \cite{Olaf:BEM}.
\end{proof}
\begin{lemma}\label{lem:jump}
	For $w\in H^{-1/2}(\Gamma_h)$ and $u = Aw$, we have the following jump relation:
	\begin{equation}\label{eq:jump}
	\gamma^{\rm int}_1 u - \gamma^{\rm ext}_1u = w.
	\end{equation}
\end{lemma}
\begin{proof}
	See Lemma 4 in \cite{Costabel}.
\end{proof}

\begin{lemma}
	The interior and exterior conormal derivatives $\gamma^{\rm int}_1: H^1(\Omega_h) \rightarrow H^{-1/2}(\Gamma_h)$ and $\gamma^{\rm ext}_1: H^1(\Omega_h^\complement) \rightarrow H^{-1/2}(\Gamma_h)$ are continuous in the sense that
	\begin{align}
	\|\gamma_1^{\rm int} u\|_{H^{-1/2}(\Gamma_h)} &\le c^{\rm int}\|\D u\|_{L^2(\Omega_h)} \label{eq: bdd_int}\\
	\|\gamma_1^{\rm ext} u\|_{H^{-1/2}(\Gamma_h)} &\le c^{\rm ext}\|\D u\|_{L^2(\Omega^\complement_h)}.\label{eq:bdd_ext}
	\end{align}

	\end{lemma}

\begin{proof}
	See Lemma 3.2 in \cite{Costabel}.
\end{proof}
\begin{proof}[Proof of Theorem \ref{thm: ellip}]
	It is clear that if $v \in H_*^{1/2}(\Gamma)$, $u  = Av(x)$ is a solution to the interior Dirichlet boundary value problem
	\begin{equation*}
	\begin{aligned}
	-\Delta u &= 0, \quad \text{in } {\Omega}_h,\\
	u & = \gamma^{\rm int}_0 (Av)(x) = (Vv)(x) , \quad \text{on } \Gamma_h.
	\end{aligned}
	\end{equation*}
	By choosing $w \in H^1(\Omega_h)$ we integrate by part to get
	\begin{equation}\label{eq:a_int}
	a_{\Omega_h}(u,w) := \int_{\Omega_h} \D u(x)\D w(x) \dx= \<\gamma^{\rm int}_1u, \gamma^{\rm int}_0 w\>_{\Gamma_h}.
	\end{equation}
	On the other hand, for $y_0\in \Omega_h$ and $R>2 {\rm diam}(\Omega_h)$, let $B_R(y_0)$ be a ball centred at $y_0$ with radius $R$. Then $u  = Av(x)$ is also the unique solution to the exterior Dirichlet boundary value problem
		\begin{equation*}
		\begin{aligned}
		-\Delta u &= 0, \quad \text{in } B_R(y_0)\setminus \bar{\Omega_h},\\
		u & = \gamma^{\rm ext}_0 (Av)(x) = (Vv)(x) , \quad \text{on } \Gamma_h,\\
		u& =  (Av)(x)  ,\quad \text{on } \pp B_R(y_0).
		\end{aligned}
		\end{equation*}
		We also integrate by part and get
	\begin{equation*}
	a_{B_R(y_0)\setminus\bar{\Omega}_h}(u,w) := \int_{B_R(y_0)\setminus\bar{\Omega}_h} \D u(x)\D w(x) \dx  =- \<\gamma^{\rm ext}_1u, \gamma^{\rm ext}_0 w\>_{\Gamma_h} + \<\gamma^{\rm int}_1u, \gamma^{\rm int}_0 w\>_{\pp B_R(y_0)}.
	\end{equation*}
	Since $u = Av(x)$ with $v \in H_*^{1/2}(\Gamma)$, by Lemma \ref{lem:decay_H*} we have
	\begin{equation*}
	|\<\gamma^{\rm int}_1u, \gamma^{\rm int}_0 u\>_{\pp B_R(y_0)}|\le C \int_{|x-y_0| = R} \frac{1}{|x-y_0|^3} \,{\rm dS} (x)\le C R^{-2} \rightarrow 0,\text{ as }R\rightarrow \infty.
	\end{equation*}
	Thus we have
	\begin{equation}\label{eq:a_ext}
	a_{\Omega^\complement_h}(u,w) = -\<\gamma^{\rm ext}_1u, \gamma^{\rm ext}_0 w\>_{\Gamma_h}.
	\end{equation}
	Consequently we have by Lemma \ref{lem:jump}
	\begin{equation}
	a_{\Omega}(u,u)  + a_{\Omega^\complement_h}(u,u)  = \< \gamma^{\rm int}_1u- \gamma^{\rm ext}_1u , \gamma^{\rm int}_0 u\>_{\Gamma_h} = \<v, \gamma_0^{\rm int}u\>_{\Gamma_h} = \<V v, v\>_{\Gamma_h}.
	\end{equation}
	Applying \eqref{eq: bdd_int} and \eqref{eq:bdd_ext} gives the ellipticity of $V$. Analogous argument follows for the ellipticity of $D$.
\end{proof}

\section{Proof of Lemma \ref{lem:g-1_rescale}}\label{app:proof_lem_rescale}

	\begin{proof}
		First we show that $u_R \in H^{1/2}_*(\Gamma_R)$. Since $u_1 \in H^{1/2}_*(\Gamma_1)$, there exists $\phi_1 \in H_*^{-1/2}(\Gamma_1)$ such that $V_1 \phi_1 = u_1$. Let $\phi_R(x) : = \tfrac{1}{R}\phi_1(x/R)$, then it is clear that  $\phi_R \in H^{-1/2}_*(\Gamma_R)$. Then we can write
		\begin{align*}
		u_R(x) = u_1(x/R) &= -\frac{1}{2\pi}\int_{\Gamma_1} \log \left|\frac{x}{R} - y\right|\phi_1(y) {\rm \, d}S(y)\\
		&= -\frac{1}{2\pi}\int_{\Gamma_R} \log \left|\frac{x}{R} - \frac{y}{R}\right|\phi_1\left(\frac{y}{R}\right)\frac{1}{R}{\rm d}S(y)\\
		& = -\frac{1}{2\pi}\int_{\Gamma_R} \log \left|\frac{x}{R} - \frac{y}{R} \right|\phi_R(y) \,{\rm d}S(y)\\
		& =  -\frac{1}{2\pi}\int_{\Gamma_R} (\log|x-y| - \log R)\phi_R(y) \,{\rm d}S(y)\\
		& =  -\frac{1}{2\pi}\int_{\Gamma_R} \log|x-y| \phi_R(y)\,{\rm d}S(y) = V_R \phi_R \in H^{1/2}_*(\Gamma_R),
		\end{align*}
		where we used the fact that $\<\phi_R,1\>_{\Gamma_R} =0$. By a similar argument of change of variables, we have
		\begin{equation*}
		[(-K_1+(1 - \lambda)I)]u_1= [-K_R+(1 - \lambda)I)]u_R.
		\end{equation*}

		Now we shall prove that $V^{-1}$ is also scale in-variant. Let $\bar{u}_1$ be the solution to the homogeneous Laplace equation
		\begin{equation*}
		\begin{aligned}
		-\Delta \bar{u}_1 &= 0, \quad \text{in } \R ^2 \setminus\Gamma_1 ,\\
		\bar{u}_1 & = u_1 , \quad \text{on } \Gamma_1,\\
		|\bar{u}_1(x)| &= \mathcal{O}\left(\frac{1}{|x|}\right) \quad \text{as } |x|\rightarrow \infty.
		\end{aligned}
		\end{equation*}
		Then $\bar{u}_R: = \bar{u}_1(x/R)$ also solves
		\begin{equation*}
		\begin{aligned}
		-\Delta \bar{u}_R &= 0, \quad \text{in } \R ^2 \setminus\Gamma_R ,\\
		\bar{u}_R & = u_R , \quad \text{on } \Gamma_R,\\
		|\bar{u}_R(x)| &= \mathcal{O}\left(\frac{1}{|x|}\right),  \quad \text{as } |x|\rightarrow \infty.
		\end{aligned}
		\end{equation*}
		For $i = 1, R$, define $v_i$ and $\bar{v}_i$ in the same way as $u_i$ and $\bar{u}_i$. Then we can apply \eqref{eq:a_int}, \eqref{eq:a_ext} and the jump relation in Lemma \ref{lem:jump} to get
		\begin{align*}
		\int_{\Omega_i}\D \bar{u}_i \D \bar{v}_i +\int_{\R ^2\setminus \Omega_i}\D \bar{u}_i \D \bar{v}_i  &=:  a_{\Omega_i(0}(\bar{u}_i,\bar{v}_i)  + a_{\R ^2\setminus \Omega_i }(\bar{u}_i,\bar{v}_i) \\
		& = \< \gamma^{\rm int}_1\bar{u}_i- \gamma^{\rm ext}_1\bar{u}_i , \gamma^{\rm int}_0 \bar{v}_i\>_{\Gamma_i} \\
		& = \<\phi_i, v_i\>_{\Gamma_i}\\
		& = \<V_i^{-1}u_i, v_i\>_{\Gamma_i}, \quad \text{for } i = 1, R,
		\end{align*}
		where $\phi_i = V_i ^{-1}u_i$.
		Clearly \begin{equation*}
		a_{\Omega_1 }(\bar{u}_1,\bar{v}_1)  + a_{\R ^2\setminus \Omega_1 }(\bar{u}_1,\bar{v}_1 ) = a_{\Omega_R }(\bar{u}_R,\bar{v}_R)  + a_{\R ^2\setminus \Omega_R }(\bar{u}_R,\bar{v}_R ).
		\end{equation*}
		Thus $\<V_1^{-1}u_1, v_1\>_{\Gamma_1 } = \<V_R^{-1}u_R, v_R\>_{\Gamma_R}$ and hence the result follows.
	\end{proof}

\bibliographystyle{plain}
\bibliography{qc}

\begin{thebibliography}{10}

\bibitem{Acosta2003}
G.~Acosta and R.~G. Duran.
\newblock An optimal poincaré inequality in l1 for convex domains.
\newblock {\em Proc. AMS 132(1)}, 195-202, 2003.

\bibitem{Adam:soblev}
R.~A. Adam.
\newblock {\em Soblev Spaces}.
\newblock Academic Press, New York, London, 1975.

\bibitem{carstensen}
C.~Carstensen.
\newblock Quasi-interpolation and a posteriori error analysis in finite element
  methods.
\newblock {\em M2AN Math. Model. Numer. Anal.}, 33:1187--1202, 1999.

\bibitem{Costabel}
M.~Costabel.
\newblock Boundary integral operators on lipschitz domains: elementary results.
\newblock {\em SIAM J. Math. Anal.}, 19(3), 1988.

\bibitem{2016-qcp2}
A.~Dedner, H.~Wu, and C.~Ortner.
\newblock Analysis of patch-test consistent atomistic-to-continuum coupling
  with higher-order finite elements.
\newblock {\em ArXiv e-prints}, 1607.05936, 2016.

\bibitem{E:2006}
W.~E, J.~Lu, and J.~Z. Yang.
\newblock {Uniform accuracy of the quasicontinuum method}.
\newblock {\em {Phys. Rev. B}}, 74(21):214115, 2006.

\bibitem{E:2005a}
W.~E and P.~Ming.
\newblock Analysis of the local quasicontinuum method.
\newblock In {\em Frontiers and prospects of contemporary applied mathematics},
  volume~6 of {\em Ser. Contemp. Appl. Math. CAM}, pages 18--32. Higher Ed.
  Press, Beijing, 2005.

\bibitem{EhrOrtSha:2013}
V.~Ehrlacher, C.~Ortner, and A.~V. Shapeev.
\newblock Analysis of boundary conditions for crystal defect atomistic
  simulations, 2013.

\bibitem{Hudson:stab}
T.~Hudson and C.~Ortner.
\newblock On the stability of {B}ravais lattices and their {C}auchy--{B}orn
  approximations.
\newblock {\em ESAIM:M2AN}, 46:81--110, 2012.

\bibitem{Kanzaki}
H.~Kanzaki.
\newblock Point defects in face-centred cubic lattice i: Distortion around
  defects.
\newblock {\em J. Phys. Chem. Solids}, 2:24–36, 1957.

\bibitem{Li:lattice_green}
X.~Li.
\newblock Boundary condition for molecular dynamics models of solids: A
  variational formulation based on lattice green's functions.
\newblock {\em preprint}.

\bibitem{2014-bqce}
X.~H. Li, C.~Ortner, A.~Shapeev, and B.~Van Koten.
\newblock Analysis of blended atomistic/continuum hybrid methods.
\newblock {\em ArXiv e-prints}, 1404.4878, 2014.

\bibitem{acta}
M.~Luskin and C.~Ortner.
\newblock Atomstic-to-continuum coupling.
\newblock {\em Acta Numerica}, 22:397 -- 508, 2013.

\bibitem{Or:2011a}
C.~{Ortner}.
\newblock {The role of the patch test in 2D atomistic-to-continuum coupling
  methods}.
\newblock {\em ESAIM Math. Model. Numer. Anal.}, 46, 2012.

\bibitem{OrShSu:2012}
C.~Ortner and A.~Shapeev.
\newblock Interpolation of lattice functions and applications to
  atomistic/continuum multiscale methods.
\newblock manuscript.

\bibitem{2013-stab.ac}
C.~Ortner, A.~Shapeev, and L.~Zhang.
\newblock (in-)stability and stabilisation of qnl-type atomistic-to-continuum
  coupling methods, 2014.

\bibitem{PRE-ac.2dcorners}
C.~{Ortner} and L.~{Zhang}.
\newblock {Construction and sharp consistency estimates for atomistic/continuum
  coupling methods with general interfaces: a 2D model problem}.
\newblock {\em SIAM J. Numer. Anal.}, 50, 2012.

\bibitem{Ortner:qnl.1d}
Christoph Ortner.
\newblock A priori and a posteriori analysis of the quasinonlocal
  quasicontinuum method in 1{D}.
\newblock {\em Math. Comp.}, 80(275):1265--1285, 2011.

\bibitem{Shimokawa:2004}
T.~Shimokawa, J.~J. Mortensen, J.~Schiotz, and K.~W. Jacobsen.
\newblock {Matching conditions in the quasicontinuum method: Removal of the
  error introduced at the interface between the coarse-grained and fully
  atomistic region}.
\newblock {\em {Phys. Rev. B}}, 69(21):214104, 2004.

\bibitem{Sinclair}
J.~E. Sinclair.
\newblock Improved atomistic model of a bcc dislocation core.
\newblock {\em Journal of Applied Physics}, 42:5321, 1971.

\bibitem{Olaf:BEM}
O.~Steinbach.
\newblock {\em Numerical Approximation Methods for Elliptic Boundary Value
  Problems: Finite and Boundary Elements}.
\newblock Springer New York, 2008.

\bibitem{Woodward}
C.~Woodward and S.~Rao.

\end{thebibliography}

\end{document}